\documentclass[a4paper,11pt,english]{smfart}
\usepackage[applemac]{inputenc}
\usepackage[francais,english]{babel}
\usepackage{amsfonts}
\usepackage{amsmath} 
\usepackage{amsthm}
\usepackage{hyperref} 
\usepackage{latexsym}
\usepackage{array}
\usepackage{mathrsfs}
\usepackage{amssymb}
\usepackage{enumerate}
\usepackage{smfthm}
\usepackage{graphicx}
\usepackage{color}
\newcommand{\ligne}{\vspace{1\baselineskip}}
\newcommand{\ph}{\phantomsection}
\newcommand{\cal}{\mathcal}

\newcommand{\R}{\mathbb  R}
\newcommand{\C}{\mathbb  C}
\newcommand{\N}{\mathbb  N}

\renewcommand{\P}{\mathbb{P}}

\newcommand{\W}{  \mathcal{W}   }

\newcommand{\M}{  \mathcal{M}  }
\newcommand{\eps}{\varepsilon}
\renewcommand{\epsilon}{\varepsilon}
\newcommand{\e}{  \text{e}   }
\newcommand{\wt}{  \widetilde   }

\renewcommand{\H}{  \mathcal{H}   }
\newcommand{\A}{  \mathcal{A}   }

\newcommand{\T}{  \S^{1} }

\newcommand{\dis}{\displaystyle}

\newcommand{\om}{  \omega   }
\newcommand{\ov}{  \overline  }
\renewcommand{\a}{  \alpha   }
\renewcommand{\b}{  \beta   }
\newcommand{\s}{  \sigma   }

\renewcommand{\phi}{  \varphi  }
\renewcommand{\L}{  \mathcal{L}   }

\renewcommand{\>}{  \rangle   }

\renewcommand{\S}{  \mathbb{S}  }

\numberwithin{equation}{section}
\theoremstyle{plain}

\newtheorem*{acknowledgements}{Acknowledgements}

\def\beq{\begin{equation}}   \def\eeq{\end{equation}}
\def\bea{\begin{eqnarray}}  \def\eea{\end{eqnarray}}

\renewcommand{\theequation}{\thesection.\arabic{equation}}
\newcounter{hran} \renewcommand{\thehran}{\thesection.\arabic{hran}}

\def\bmini{\setcounter{hran}{\value{equation}}
    \refstepcounter{hran}\setcounter{equation}{0}
    \renewcommand{\theequation}{\thehran\alph{equation}}\begin{eqnarray}}

\def\bminiG#1{\setcounter{hran}{\value{equation}}
\refstepcounter{hran}\setcounter{equation}{-1}
\renewcommand{\theequation}{\thehran\alph{equation}}
\refstepcounter{equation}\label{#1}\begin{eqnarray}}

\author{ Aur\'elien Poiret}
\address{Laboratoire de Math\'ematiques, UMR 8628 du CNRS.
Universit\'e Paris Sud, 91405 Orsay Cedex, France}
\email{aurelien.poiret@math.u-psud.fr}
\author{ Didier Robert}
\address{Laboratoire de Math\'ematiques J. Leray, UMR  6629 du CNRS, Universit\'e de Nantes, 
2, rue de la Houssini\`ere,
44322 Nantes Cedex 03, France}
\email{didier.robert@univ-nantes.fr}
\author{ Laurent Thomann }
\address{Laboratoire de Math\'ematiques J. Leray, UMR  6629 du CNRS, Universit\'e de Nantes, 
2, rue de la Houssini\`ere,
44322 Nantes Cedex 03, France}
\email{laurent.thomann@univ-nantes.fr}

 \title[Probabilistic global well-posedness for NLS]{Probabilistic global well-posedness for the supercritical nonlinear harmonic oscillator}
 
\begin{document}
\frontmatter

 \begin{abstract}
Thanks to an approach inspired from  Burq-Lebeau \cite{bule}, we prove stochastic versions of Strichartz estimates for Schr\"odinger with harmonic potential. As a consequence, we show that the nonlinear Schr\"odinger equation with quadratic potential  and any polynomial non-linearity  is almost surely locally well-posed in $L^{2}(\R^{d})$ for any $d\geq 2$. Then, we show that we can combine this result with the high-low frequency decomposition method of Bourgain to prove      a.s.  global well-posedness results for the cubic equation: when  $d=2$, we prove global well-posedness in $\H^{s}(\R^{2})$ for any $s>0$, and when $d=3$ we  prove global well-posedness in $\H^{s}(\R^{3})$ for any $s>1/6$, which is a supercritical regime. 

Furthermore, we also obtain almost sure global well-posedness results with scattering  for  NLS on $\R^{d}$ without  potential. We prove scattering results  for $L^2-$supercritical equations and $L^2-$subcritical equations with  initial conditions in $L^2$ without additional decay or regularity assumption.
  \end{abstract}
\thanks{\noindent D. R.  was partly supported  by the  grant  ``NOSEVOL''   ANR-2011-BS01019 01.\\
L.T. was partly supported  by the  grant  ``HANDDY'' ANR-10-JCJC 0109 \\
 \indent \quad and by the  grant  ``ANA\'E'' ANR-13-BS01-0010-03.}
\maketitle
\mainmatter

 \section{Introduction and results}
 
  \subsection{Introduction}
  It is known from  several works  that a probabilistic   approach can  help to give insight in dynamics of dispersive non linear  PDEs, even for low  Sobolev regularity. This point of view was initiated by Lebowitz-Rose-Speer  \cite{LeRoSp},  developed by   Bourgain \cite{Bourgain1,Bourgain2} and Zhidkov
\cite{Zhidkov}, and enhanced by Tzvetkov~\cite{Tzvetkov1,Tzvetkov2,Tzvetkov3},
Burq-Tzvetkov \cite{BT2,BT3}, Oh \cite{Oh1,Oh2}, Colliander-Oh \cite{CO} and others.  In this paper we study the Cauchy problem for the nonlinear  Schr\"odinger-Gross-Pitaevskii equation
 \begin{equation} \label{SH} 
  \left\{
      \begin{aligned}
         & i \frac{ \partial u }{ \partial t } +\Delta u-|x|^{2}u = \pm |u|^{p-1} u, \quad (t,x)\in \R\times \R^{d},
       \\  &  u(0)  =u_0,
      \end{aligned}
    \right.
\end{equation}
with  $d\geq 2$, $ p \geq 3 $  an odd integer and where $u_{0}$ is a random initial condition.

Much work has been done on dispersive PDEs with random initial conditions since the papers of Burq-Tzvetkov \cite{BT2,BT3}. In these articles, the authors showed that thanks to a randomisation of the initial condition one can prove well-posedness results even for data with supercritical Sobolev regularity. We also refer to   \cite{BTproba}, Thomann \cite{tho}, Burq-Thomann-Tzvetkov \cite{BTT}, Poiret \cite{poiret1,poiret2}, Suzzoni~\cite{deSuzzoni}  and Nahmod-Staffilani \cite{NaSt} for strong solutions in a probabilistic sense. Concerning  weak solutions, see~\cite{BTT2,BTT3} as well as
Nahmod-Pavlovic-Staffilani~\cite{NPS}. 

More recently, Burq-Lebeau~\cite{bule}  considered a different randomisation method, and thanks to fine spectral estimates, they obtained better stochastic bounds which enabled them to improve the previous known results for the supercritical wave equation on a compact manifold. In \cite{PRT1} we extended the results of~\cite{bule} to the harmonic oscillator in $\R^{d}$. This approach enables to prove a stochastic version of the usual Strichartz estimates with a gain of $d/2$ derivatives, which we will use here to apply to the nonlinear problem. These estimates (the result of Proposition \ref{regularite}) can be seen as a consequence of~\cite[Inequality~(1.6)]{PRT1}, but we give here an alternative proof suggested by Nicolas Burq. \ligne

Consider a  probability space
$(\Omega, {\cal F}, { \P})$ and let  $\{g_n\}_{n\geq 0}$  be a sequence of real random variables, which we will assume to be independent and identically distributed.   We assume that   the common law $\nu$ of~$g_{n}$ satisfies for some   $c>0$ the bound
\beq\label{expo}
\int_{-\infty}^{+\infty}{\rm e}^{\gamma x}\,\text{d}\nu \leq {\rm e}^{c\gamma^2},\;\;\;\;\forall \gamma\in\R.
\eeq
This condition implies in particular that the  $g_n$ are centred variables. It is easy to check that \eqref{expo}  is satisfied for centred Gauss laws  and for any centred  law   with bounded support. Under condition~\eqref{expo}, we can prove the Khinchin inequality (Lemma \ref{Khin}) which we will use in the sequel.\ligne

   Let $d\geq2$.  We denote by 
     $$H=-\Delta +|x|^{2},$$
     the harmonic oscillator and by $\{\phi_{j},\;j\geq 1\}$ an orthonormal basis of $L^{2}(\R^{d})$ of eigenvectors  of~$H$ (the Hermite functions).   The eigenvalues of $H$ are the $\big\{2(\ell_{1}+\dots+\ell_{d})+d,\; \ell\in \N^{d}\big\}$, and we can order them in a  non decreasing sequence $\{\lambda_j, \;j \geq1\}$, repeated according to their multiplicities, and so that  $H \phi_{j} =\lambda_j \phi_{j}$.   
     
   We define the  harmonic Sobolev spaces for $s\geq 0$, $p\geq 1$ by 
 \begin{equation*} 
         \W^{s, p}= \W^{s, p}(\R^d) = \big\{ u\in L^p(\R^d),\; {H}^{s/2}u\in L^p(\R^d)\big\},     
       \end{equation*}
       \begin{equation*}
           \H^{s}=   {\cal H}^{s}(\R^d) = \W^{s, 2}.
       \end{equation*}
             The natural norms are denoted by $\Vert u\Vert_{\W^{s,p}}$ and up to equivalence of norms  (see \cite[Lemma~2.4]{YajimaZhang2}), for $1<p<+\infty$,  we have
        $$
      \Vert u\Vert_{\W^{s,p}} = \Vert  H^{s/2}u\Vert_{L^{p}} \equiv \Vert (-\Delta)^{s/2} u\Vert_{L^{p}} + 
       \Vert\<x\>^{s}u\Vert_{L^{p}}.
       $$
For $j\geq 1$ denote by 
     \begin{equation*} 
     I(j)=\big\{n\in \N,\;2j\leq \lambda_{n}< 2(j+1)\big\}.
     \end{equation*}
     Observe that for all $j \geq d/2$, $I(j)\neq \emptyset$ and that $\#I(j)\sim c_{d}j^{d-1}$ when $j \longrightarrow +\infty$.\ligne

    Let $s\in \R$, then any $u\in \H^{s}(\R^{d})$ can be written in a unique fashion 
   $$u=\sum_{j=1}^{+\infty} \sum_{n\in I(j)}c_{n}\phi_{n}.$$
Following a suggestion of Nicolas Burq, we introduce the following condition  
    \begin{equation}\label{condi0} 
  |c_{k}|^{2}\leq \frac{C}{\# I(j)}\sum_{n\in I(j)}|c_{n}|^{2},\quad \forall k\in I(j),\quad \forall  j\geq 1,
  \end{equation}
  which  means that the coefficients have almost the same size on each level of energy $I(j)$.   Observe that this condition is always satisfied in dimension $d=1$. We define the set $\mathcal{A}_{s}\subset \H^{s}(\R^{d})$ by     
      \begin{equation*}
      \mathcal{A}_{s}=\big\{u=\sum_{j=1}^{+\infty} \sum_{n\in I(j)}c_{n}\phi_{n}\in \H^{s}(\R^{d})\;\;    \text{s.t. condition \eqref{condi0} holds for some}\; C>0\big\}.
      \end{equation*}

      It is easy to check the following properties \vspace{5pt}
      
      \begin{itemize}
      \item[$\bullet$] Let $u\in \A_{s}$, then  for all $c\in \C$, $c u \in \A_{s}$.
      \item[$\bullet$] The set $\A_{s}$ is neither closed nor open in $\H^{s}$.
      \item[$\bullet$] The set $\A_{s}$ is invariant under the linear Schr\"odinger flow $\e^{-itH}$.
       \item[$\bullet$] The set $\A_{s}$ depends on the choice of the orthonormal basis $(\phi_{n})_{n\geq 1}$. Indeed, given $u\in \H^{s}$, it is easy to see that there exists a Hilbertian basis $(\tilde{\phi}_{n})_{n\geq 1}$ so that $u\in \tilde{\A}_{s}$, where $\tilde{\A}_{s}$ is  the  space based on~$(\tilde{\phi}_{n})_{n\geq 1}$.
      \end{itemize}
~

 Let $\gamma \in \A_{s}$. We define the probability measure $\mu_{\gamma}$ on $\H^{s}$ via the map 
 \begin{equation*}
 \begin{array}{rcl}
\Omega&\longrightarrow&\H^{s}(\R^{d})\\[3pt]
\dis  \omega&\longmapsto &\dis\gamma^{\om}=\sum_{j=1}^{+\infty} \sum_{n\in I(j)}c_{n}g_{n}(\om)\phi_{n},
 \end{array}
 \end{equation*}
in other words, $\mu_{\gamma}$ is defined by: for all measurable $F: \H^{s}\longrightarrow \R$
 \begin{equation*}
 \int_{\H^{s}(\R^{d})}F(v)\text{d}\mu_{\gamma}(v)=\int_{\Omega}F(\gamma^{\om})\text{d}\P(\om).
 \end{equation*}
 In particular, we   can check that $\mu_{\gamma}$ satisfies \vspace{5pt}
      
      \begin{itemize}
      \item[$\bullet$] If $\gamma \in \H^{s}\backslash \H^{s+\eps}$, then $\mu_{\gamma}(\H^{s+\eps})=0$.
      \item[$\bullet$] Assume that for all $j\geq 1$ such that $I(j)\neq \emptyset$ we have $c_{j}\neq 0$. Then for all nonempty open subset $B\subset \H^{s}$, $\mu_{\gamma}(B)>0$.
      \end{itemize}
      ~
      
  Finally, we denote by $\M^{s}$    the set of all such measures
  \begin{equation*}
  \M^{s}=\bigcup_{\gamma \in \A_{s}}\{\mu_{\gamma}\}.
  \end{equation*}

  \subsection{Main results}  Before we state our results, let us recall some facts concerning the deterministic study of the nonlinear Schr\"odinger equation \eqref{SH}. We say that \eqref{SH} is locally well-posed in $\H^{s}(\R^{d})$, if for all initial condition $u_{0}\in \H^{s}(\R^{d})$, there exists a unique local in time solution $u\in \mathcal{C}([-T,T];\H^{s}(\R^{d}))$, and if the flow-map is uniformly continuous. We denote by 
  \begin{equation*}
  s_{c}=\frac{d}2-\frac{2}{p-1},
  \end{equation*}
   the critical Sobolev index. Then one can show that NLS  is well-posed in $\H^{s}(\R^{d})$ when $s>\max(s_{c},0)$, and ill-posed when $s<s_{c}$. We refer to the introduction of \cite{tho} for more details on  this topic. 
 
 \subsubsection{\bf Local existence results}
We are now able to state our first result on the local well-posedness of \eqref{SH}.

 \begin{theo}\ph \label{letheoreme}
Let   $ d \geq 2 $, $ p \geq 3$ an odd integer and fix $ \mu=\mu_{\gamma}\in \M^{0}$. Then there exists $\Sigma\subset L^{2}(\R^{d})$ with $\mu(\Sigma)=1$ and so that:
\begin{enumerate}[(i)]
\item For all $u_{0}\in \Sigma$ there exist $T>0$ and  a unique local solution $u$ to \eqref{SH} with initial data $u_{0}$ satisfying
\begin{equation}\label{solu} 
u(t) - \e^{-itH} u_0  \in       \mathcal{C}\big([-T,T]; \H^{s}(\R^{d})\big),
\end{equation}
for some $\frac{d}{2}-\frac{2}{p-1}<s<\frac{d}{2}$.
 \item More precisely, for all $T>0$, there exists $\Sigma_{T}\subset \Sigma$ with 
 \begin{equation*} 
 \mu(\Sigma_{T})\geq 1-C\exp\big(-cT^{-\delta}\|\gamma\|^{-2}_{L^{2}(\R^{2})}\big),\quad C,c,\delta > 0,
 \end{equation*}
 and such that for all $u_{0}\in \Sigma_{T}$ the lifespan of $u$ is larger than $T$.
 \end{enumerate}
\end{theo}

Denote by $\dis \gamma=\sum_{n=0}^{+\infty} c_{n}\phi_{n}(x)$, then $\dis u^{\om}_{0}:=\sum_{n=0}^{+\infty} g_{n}(\om)c_{n}\phi_{n}(x)$ is a typical element in the support of $\mu_{\gamma}$.  Another way to state  Theorem \ref{letheoreme} is :  for any $ T>0 $, there exists an event $ \Omega_T \subset \Omega$ so that 
\begin{equation*}
{\P}( \Omega _T ) \geq 1-C\exp\big(-cT^{-\delta}\|\gamma\|^{-2}_{L^{2}(\R^{d})}\big),\quad C,c,\delta > 0,
\end{equation*}
and so that for all $ \omega \in \Omega _T $, there exists a unique solution  of the form \eqref{solu} to (\ref{SH}) with initial data~$ u_0^\omega $.\ligne

We will see in Proposition \ref{regularite}  that the stochastic approach yields a gain of $d/2$ derivatives compared to the deterministic theory. To prove       Theorem \ref{letheoreme} we only have to gain  $s_{c}=d/2-2/(p-1)$ derivatives. The solution is constructed by a fixed point argument in a Strichartz space $X^s_{T} \subset \mathcal{C}\big([-T,T];\H^{s}(\R^{d})\big)$ with continuous embedding, and uniqueness holds in the class $X^{s}_{T}$. \ligne
 
 The deterministic Cauchy problem for \eqref{SH} was studied by Oh \cite{Oh89} (see also  Cazenave~\cite[Chapter~9]{Caze} for more references).
In~\cite{tho}, Thomann has proven an almost sure local existence result for~\eqref{SH} in the supercritical regime (with a gain of 1/4 of derivative), for any $d\geq 1$. This local existence result was improved by Burq-Thomann-Tzvetkov~\cite{BTT}  when $d=1$  (gain of 1/2 derivatives), by Deng~\cite{Deng} when $d=2$, and by Poiret~\cite{poiret1,poiret2} in any dimension.
\begin{rema}
The results of Theorem \ref{letheoreme} also hold true for any quadratic potential 
$$V(x)=\sum_{1\leq j\leq d}\alpha_jx_j^2,\quad
\alpha_j>0,\quad 1\leq j\leq d,
$$
and for more general potentials such that $V(x)\approx \<x\>^{2}$.
\end{rema}
 \subsubsection{\bf Global existence and scattering results for NLS} As an application of the results of the previous part, we are able to construct global solutions  to the non-linear Schr\"odinger equation without potential, which scatter when $t\to \pm \infty$. Consider the following equation
 \begin{equation} \label{NLS} 
  \left\{
      \begin{aligned}
         & i \frac{ \partial u }{ \partial t } +\Delta u = \pm  |u|^{p-1} u, \quad (t,x)\in \R\times \R^{d}.
       \\  &  u(0)  =u_0.
      \end{aligned}
    \right.
\end{equation}
The well-posedness indexes for this equation are the same as for equation \eqref{SH}. Namely, \eqref{NLS}  is well-posed in $H^{s}(\R^{d})$ when $s>\max(s_{c},0)$, and ill-posed when $s<s_{c}$.

For the next result, we will need an additional condition on the law $\nu$. We assume that 
\begin{equation}\label{petit}
\P\big(|g_{n}|<\rho \big)>0,\quad \forall\,\rho >0,
\end{equation}
which ensure that the r.v. can take arbitrarily small values.  Then we can prove
\begin{theo}\ph\label{Thm12}
Let   $ d \geq 2 $, $ p \geq 3$ an odd integer and fix $ \mu=\mu_{\gamma}\in \M^{0}$. Assume that \eqref{petit} holds. Then there exists $\Sigma\subset L^{2}(\R^{d})$ with $\mu(\Sigma)>0$ and so that:
\begin{enumerate}[(i)]
\item For all $u_{0}\in \Sigma$ there exists a unique global solution $u$ to \eqref{NLS} with initial data $u_{0}$ satisfying
\begin{equation*} 
u(t) - \e^{it\Delta} u_0  \in       \mathcal{C}\big(\R; \H^{s}(\R^{d})\big),
\end{equation*}
for some $\frac{d}{2}-\frac{2}{p-1}<s<\frac{d}{2}$.
 \item For all $u_{0}\in \Sigma$ there exist states $f_{+},f_{-}\in \H^{s}(\R^{d})$ so that when $t\longrightarrow \pm \infty$
  \begin{equation*}
\big\|u(t) - \e^{it\Delta}( u_0+f_{\pm})\big\|_{H^{s}(\R^{d})}\longrightarrow 0.
 \end{equation*}
 \item If we assume that the distribution of $\nu$ is  symmetric, then
 $$\mu\Big(u_{0}\in L^{2}(\R^{d}) : {\text the \;assertion\; {\it (ii)}\; holds\; true}\; \Big|   \;   \|u_{0}\|_{L^{2}(\R^{d})}\leq \eta \Big)\longrightarrow 1,$$ 
 when $\eta \longrightarrow 0$.
 \end{enumerate}
\end{theo}
We can show   \cite[Théorème 20]{poiret1}, that for all $ s> 0 $,  if $ u_0 \notin \H^\s(\R^d) $ then $\mu(H^{\s}(\R^{d}))=0$.
This  shows that the randomisation does not yield a gain of derivative in the Sobolev scale;  thus Theorem~\ref{Thm12}  gives results for initial conditions
which are not covered by the deterministic theory.\ligne

There is a large litterature for the deterministic local and global theory with scattering for \eqref{NLS}. We refer to \cite{BaCaSt, NaOz,Carles2} for such results and more  references.\ligne

 We do not give here the details of the proof of Theorem \ref{Thm12}, since one can follow the main lines of the argument of Poiret   \cite{poiret1, poiret2} but with a different numerology (see {\it e.g. \cite[Théorème 4]{poiret2}}). The proof of $(i)$ and $(ii)$ is based on the use of an explicit transform, called the lens  transform $\mathscr{L}$, which links the solutions to~\eqref{NLS} to solutions of NLS with harmonic potential. 
 The transform $ \mathscr{L} $ has been used in different contexts: see Carles \cite{Carles2} for scattering results and more references.   
 More precisely, for  $ \dis u(t,x):  ] - \frac{\pi}{4} ; \frac{\pi}{4}  [  \times \R^d \longrightarrow \C$ we define
\begin{equation*}
v(t,x) = \mathscr{L}u(t,x)= \left( \frac{1}{\sqrt{1+4t^2}}  \right) ^{d/2}   u \Big( \frac{\arctan(2t)}{2}  , \frac{x}{\sqrt{1+4t^2} } \Big)     \e^{ \frac{ix^2t}{1+4t^2}  },
\end{equation*}
then $u$ is solution to
 $ \dis i \frac{ \partial u }{ \partial t } -H u  =   \lambda \cos(2t)^{  \frac{d}{2}(p-1) -2 } | u|^{p-1} u $ if and only if $v$ satisfies $ \dis i \frac{ \partial v}{ \partial t } + \Delta v   = \lambda | v|^{p-1} v$.
  Theorem \ref{letheoreme} provides  solutions  with lifespan larger than $\pi/4$  for large probabilities, provided that  the  initial conditions are small enough.

 The point  $(iii)$ is stated in  \cite[Théorème 9]{poiret1}, and can be understood as a small data result.
  \ligne 
 
 In Theorem \ref{Thm12} we assumed that $d\geq 2$ and $p\geq 3$ was an odd integer, so we had $p\geq 1+4/d$, or in other words we were in a $L^{2}$-supercritical setting. Our approach also allows to get results in an $L^{2}$-subcritical context, i.e. when $1+2/d<p< 1+4/d$.

 \begin{theo}\ph\label{Thm14}
Let $d=2$ and $2<p<3$. Assume that \eqref{petit} holds and fix $ \mu=\mu_{\gamma}\in \M^{0}$.  Then there exists $\Sigma\subset L^{2}(\R^{2})$ with $\mu(\Sigma)>0$ and so that for all $0<\eps<1$
\begin{enumerate}[(i)]
\item For all $u_{0}\in \Sigma$ there exists a unique global solution $u$ to \eqref{NLS} with initial data $u_{0}$ satisfying
\begin{equation*} 
u(t) - \e^{it\Delta} u_0  \in       \mathcal{C}\big(\R; \H^{1-\eps}(\R^{2})\big).
\end{equation*}
 \item For all $u_{0}\in \Sigma$ there exist states $f_{+},f_{-}\in \H^{1-\eps}(\R^{2})$ so that when $t\longrightarrow \pm \infty$
  \begin{equation*}
\big\|u(t) - \e^{it\Delta}( u_0+f_{\pm})\big\|_{H^{1-\eps}(\R^{2})}\longrightarrow 0.
 \end{equation*}
 \item If we assume that the distribution of $\nu$ is  symmetric, then
 $$\mu\Big(u_{0}\in L^{2}(\R^{2}) : {\text the \;assertion\; {\it (ii)}\; holds\; true}\; \Big|   \;   \|u_{0}\|_{L^{2}(\R^{2})}\leq \eta \Big)\longrightarrow 1,$$ 
 when $\eta \longrightarrow 0$.
 \end{enumerate}
\end{theo}

In the case $p\leq1+2/d$, Barab \cite{Barab} showed that a non trivial solution to \eqref{NLS} never scatters, therefore even with a stochastic approach one can not have scattering in this case. When $d=2$, the condition~$p>2$ in Theorem \ref{Thm14} is therefore optimal.
Usually, deterministic scattering results in  $L^2$-subcritical contexts are obtained in the space $H^1 \cap \mathcal{F}(H^1)$. Here we assume $u_0\in L^2$, and thus   we relax both the regularity and the decay assumptions (this latter point is the most striking in this context). Again we refer to \cite{BaCaSt} for an overview of scattering theory for NLS.\medskip

When $\mu\in \M^{\s}$ for some $0<\s<1$ we are able to prove the same result with $\eps=0$. Since the proof is much easier, we give it before the case $\s=0$ (see Section \ref{sect32}).\medskip

Finally, we point out that in Theorem \ref{Thm14} we are only able to consider the case $d=2$ because of the lack of regularity of the nonlinear term $|u|^{p-1}u$.

  \subsubsection{\bf Global existence results for NLS with quadratic potential}

We also get global existence results for defocusing Schr\"odinger equation with harmonic potential. For $ d=2 $ or $ d=3 $, consider the  equation 
\begin{equation} \label{NLS3D}  
  \left\{
      \begin{aligned}
         & i \frac{ \partial u }{ \partial t } - H u =   |u|^{2} u, \quad (t,x)\in \R\times \R^{d},
       \\  &  u(0)  =u_0,
      \end{aligned}
    \right.
\end{equation}
and denote by $E$ the energy of \eqref{NLS3D}, namely
\begin{equation*}
E(u)=\|  u\|^{2}_{\H^{1}(\R^{d})}+\frac12 \|u\|^{4}_{L^{4}(\R^{d})}.
\end{equation*}
Deterministic global existence for \eqref{NLS3D} has been studied by Zhang \cite{Zhang05} and  by Carles \cite{Car11} in  the case of time-dependent potentials.\ligne 

When $d=3$, our global existence  result for \eqref{NLS3D} is the following

\begin{theo}\ph\label{theo3D}
Let $ d=3 $, $1/6<s< 1$ and fix $\mu=\mu_{\gamma} \in \mathcal{M}^{s}$. Then there exists a set $\Sigma\subset \H^{s}(\R^{3})$ so that $\mu(\Sigma)=1$ and so that the following holds true
\begin{enumerate}[(i)]
\item For all $u_{0}\in \Sigma$, there exists a unique global solution to \eqref{NLS3D}   which reads
\begin{equation*}
u(t)=\e^{-itH}u_{0}+w(t),\quad w\in \mathcal{C}\big(\R,\H^{1}(\R^{3})\big).
\end{equation*}
\item The previous line defines a global flow $\Phi$, which leaves the set $\Sigma$ invariant
\begin{equation*}
\Phi(t)(\Sigma)=\Sigma, \quad \text{for all}\;\; t\in \R.
\end{equation*}
\item There exist $C,c_{s}>0$ so that for all $t\in \R$, 
\begin{equation*}
E\big(w(t)\big)\leq C(M+|t|)^{c_{s}+},
\end{equation*}
where $M$ is a positive random variable so that 
$$\mu(u_{0}\in \H^{s}(\R^{3})\,: M>K)\leq C\e^{-\frac{cK^{\delta}}{\|\gamma\|^{2}_{\H^{s}(\R^{3})}}}.$$
\end{enumerate}
\end{theo}

Here the critical Sobolev space is $\H^{1/2}(\R^3)$, thus the local deterministic theory combined with the conservation of the energy, immediately gives global well-posedness in $\H^{1}(\R^{3})$. Using a kind of interpolation method due to Bourgain, one may obtain deterministic global well-posedness in $\H^{s}(\R^{3})$ for some $1/2<s<1$. Instead, for the proof of Theorem \ref{theo3D}, we will rely on the almost well-posedness result of Theorem \ref{letheoreme}, and this gives  global well-posedness  in a supercritical context. \medskip

The constant $c_{s}>0$ can be computed explicitly (see \eqref{defcs}), and we do not think that we have obtained the optimal rate. By reversibility of the equation, it is enough to consider only positive times.\ligne

With a similar approach, in dimension $d=2$, we can prove the following result

\begin{theo}\ph\label{theo2D}
Let $ d=2 $, $0<s< 1$ and fix $\mu=\mu_{\gamma} \in \mathcal{M}^{s}$. Then there exists a set $\Sigma\subset \H^{s}(\R^{2})$ so that $\mu(\Sigma)=1$ and so that for  all $u_{0}\in \Sigma$, there exists a unique global solution to \eqref{NLS3D}   which reads
\begin{equation*}
u(t)=\e^{-itH}u_{0}+w(t),\quad w\in \mathcal{C}\big(\R,\H^{1}(\R^{2})\big).
\end{equation*}
In addition, statements $(ii)$ and $ (iii) $ of Theorem \ref{theo3D} are also satisfied with $c_{s}=\frac{1-s}s$.
\end{theo}

Here the critical Sobolev space is $L^2(\R^2)$, thus Theorem \ref{theo2D} shows global well-posedness for any subcritical cubic non linear Schr\"odinger equations in dimension two.\ligne

Using the smoothing effect which yields a gain of 1/2 derivative,  a global well-posedness result for~\eqref{SH}, in the defocusing case, was given in  \cite{BTT} in the case $d=1$, for any $p\geq 3$. The global existence is proved for a typical initial condition on the support of a Gibbs measure, which is $\dis \cap_{\s>0}\H^{-\s}(\R)$. This result was extended by Deng \cite{Deng} in dimension $d=2$ for radial functions. However, this approach has the drawback  that it relies on the invariance of a Gibbs measure, which is a rigid object, and is supported in  rough Sobolev spaces. Therefore it seems difficult to adapt this strategy in higher dimensions. 

Here instead we obtain the results of Theorems \ref{theo3D} and \ref{theo2D} as a combination of   Theorem \ref{letheoreme} with the high-low frequency decomposition method of Bourgain \cite[page 84]{Bou}. This approach has been successful in different contexts, and has been first used together with probabilistic arguments by Colliander-Oh~\cite{CO} for the cubic Schr\"odinger below $L^{2}(\S^{1})$ and later on by Burq-Tzvetkov \cite{BTproba} for the wave equation.

 \subsection{Notations and plan of the paper} 
 
 \begin{enonce*}{Notations}
 In this paper $c,C>0$ denote constants the value of which may change
from line to line. These constants will always be universal, or uniformly bounded with respect to the other parameters.\\ 
We will sometimes use the notations $L^{p}_{T}=L^{p}_{[-T,T]}=L^{p}(-T,T)$ for $T>0$ and we write $L^{p}_{x}=L^{p}(\R^{d})$. We denote by $H=-\Delta+|x|^{2}=\sum_{j=1}^{d}(-\partial^{2}_{j}+x^{2}_{j})$ the harmonic oscillator on $\R^{d}$, and   for $s\geq 0$ we define the   Sobolev space $\H^{s}$  by the norm  $\|u\|_{\H^{s}}=\|H^{s/2}u\|_{L^{2}(\R^{d})}$.  More generally, we define the spaces $\W^{s,p}$ by the norm $\|u\|_{\W^{s,p}}=\|H^{s/2}u\|_{L^{p}(\R^{d})}$.
 If $E$ is a Banach space and  $\mu$ is a measure on $E$, we write  $L^{p}_{\mu}=L^{p}(\text{d}\mu)$ and $\|u\|_{L^{p}_{\mu}E}=\big\|\|u\|_{E}\big\|_{L^{p}_{\mu}}$. 
\end{enonce*}

The rest of the paper is organised as follows. In Section \ref{Sect2} we recall some deterministic results on the spectral function, and  prove  stochastic Strichartz estimates. Section \ref{Sect3}  is devoted to the proof of  Theorem \ref{letheoreme} and of the scattering results for NLS without potential. Finally, in Section  \ref{Sect4} we study the global existence for the Schr\"odinger-Gross-Pitaevskii equation \eqref{SH}.    
 
\begin{acknowledgements}
We are grateful to Nicolas Burq for discussions on this subject. We thank R\'emi Carles for discussions on scattering theory which led  us  toward Theorem~\ref{Thm14}.
\end{acknowledgements}


\section{Stochastic Strichartz estimates}\label{Sect2}

The main result of this section is the following probabilistic improvement of the Strichartz estimates.
                \begin{prop}\ph  \label{regularite}
   Let $s\in \R$ and $\mu=\mu_{\gamma}\in \M^{s}$. Let $1\leq q<+\infty$, $2\leq r\leq +\infty$, and set $\a=d(1/2-1/r)$ if $r<+\infty$ and $\a<d/2$ if $r=+\infty$. Then  there exist $c,C>0$ so that for all $\tau \in \R$
   \begin{equation*}
   \mu\big(u\in \H^{s}(\R^{d})\,:\;\; \big\|  \e^{-i(t+\tau)H}u  \big\|_{L^{q}_{[0,T]}\W^{s+\a,r}(\R^{d})}>K\big)\leq C\e^{-\frac{cK^{2}}{T^{2/q}\|\gamma\|^{2}_{\H^{s}(\R^{d})}}}.
   \end{equation*}
     \end{prop}
When $r=+\infty$, this result expresses a gain $\mu-$a.s. of $d/2$ derivatives  in space compared to the deterministic Strichartz estimates (see the bound \eqref{Stri0}).\ligne

Proposition \ref{regularite} is  a consequence of~\cite[Inequality~(1.6)]{PRT1}, but we give here a self contained proof  suggested by Nicolas Burq.\ligne

There are two key ingredients in the proof of Proposition \ref{regularite}. The first one is a deterministic estimate on the spectral function given in Lemma \ref{lemme5}, and the second is  the Khinchin inequality stated in Lemma \ref{Khin}.

\subsection{Deterministic estimates of the spectral function}
 We define the spectral function $\pi_{H}$ for the harmonic oscillator by 
  \begin{equation*}
  \pi_{H}(\lambda;x,y) = \sum_{\lambda_j\leq \lambda}\phi_{j}(x)\overline{\phi_{j}(y)},
  \end{equation*}
  and this definition does not depend on the choice of $\{\phi_{j}, j\in\N\}$.
   
  Let us recall some results of $\pi_{H}$, which were essentially obtained by 
  Thangavelu \cite[Lemma 3.2.2, p. 70]{Thanga} (see also  Karadzhov \cite{ka1} and  \cite[Section 3]{PRT1} for more details).\ligne
  
  Thanks to the Mehler formula, we can prove
\begin{equation}\label{karest}
\pi_H(\lambda; x,x) \leq C\lambda^{d/2}\exp\Big(-c\frac{\vert x\vert^2}{\lambda}\Big),\quad
\forall x\in\R^d, \;\lambda \geq 1.
\end{equation}
One also have the following more subtle bound, which is the heart of the work \cite{ka1}
  \beq\label{kara3}
\vert\pi_H(\lambda +\mu;x,x) - \pi_H(\lambda ;x,x)\vert \leq C(1+\vert\mu\vert)\lambda^{d/2 -1},\quad
\lambda \geq 1,\; \vert\mu\vert \leq C_0\lambda.
\eeq
This inequality gives of bound on $\pi_{H}$ in energy interval of size $\sim$1, which is the finest one can obtain.\ligne

Then we can prove (see \cite[Lemma 3.5]{PRT1})
\begin{lemm}\ph \label{lemme5}
Let $d\geq 2$ and assume that $|\mu|\leq c_{0}$, $r\geq 1$ and $\theta \geq 0$. Then  there exists $C>0$ so that for all~$\lambda\geq 1$
\begin{equation*}
\|\pi_{H}(\lambda+\mu;x,x)-\pi_{H}(\lambda;x,x)\|_{L^{r}(\R^{d})} \leq C \lambda ^{  \frac{d}2(1+\frac1r)-1  }.
\end{equation*}
\end{lemm}
  \subsection{Proof of Proposition \ref{regularite}}
       
        To begin with, recall the Khinchin inequality which shows a smoothing property of the random series in 
the $L^{k}$ spaces for $k\geq 2$. See {\it e.g.} \cite[Lemma 4.2]{BT2} for the  proof.
\begin{lemm}\ph\label{Khin}
There exists $C>0$ such that for all real $k\geq 2$ and $(c_{n})\in \ell^{2}(\N)$
\begin{equation*} 
\|\sum_{n\geq 1}g_{n}(\om)\,c_{n}\|_{L_{\P}^{k}}\leq C\sqrt{k}\Big(\sum _{n\geq 1}|c_{n}|^{2}\Big)^{\frac12}.
\end{equation*}
\end{lemm}

Now we fix $\dis\gamma=\sum_{n=0}^{+\infty}c_{n} \phi_{n}\in \A_{s}$ and denote by $\dis \gamma^{\om}=\sum_{n=0}^{+\infty}g_{n}(\omega)c_{n}\phi_{n}$. 

$\bullet$ Firstly, we treat the case $r<+\infty$. Set  $\dis \a = d(1/2-1/r)$ and set $\s=s+\a$. Observe that it suffices to prove the estimation for $ K \gg \|\gamma\|_{{\mathcal H}^{s}(\R^{d})} $.

Let $k\geq 1$, then  by definition 
\begin{eqnarray}
\int_{\H^{s}(\R^{d})}\big\|  \e^{-i(t+\tau)H}u  \big\|^{k}_{L^{q}_{[0,T]}\W^{\s,r}(\R^{d})}\text{d}\mu(u)&=&\int_{\Omega}\big\|  \e^{-i(t+\tau)H}\gamma^{\om}  \big\|^{k}_{L^{q}_{[0,T]}\W^{\s,r}(\R^{d})}\text{d}\P(\om) \nonumber\\
&=&\int_{\Omega}\big\|  \e^{-i(t+\tau)H}H^{\s/2}\gamma^{\om}  \big\|^{k}_{L^{q}_{[0,T]}L^{r}(\R^{d})}\text{d}\P(\om).\label{mu.defi}
\end{eqnarray}
Since $\e^{-i(t+\tau)H}H^{\s/2}\gamma^{\om}(x)=\sum_{n=0}^{+\infty}g_{n}(\omega)c_{n}\lambda^{\s/2}_{n}\e^{-i(t+\tau)\lambda_{n}}\phi_{n}(x)$, by Lemma \ref{Khin} we get
 \begin{equation*}
\|\e^{-i(t+\tau)H}H^{\s/2}\gamma^{\om}(x)\|_{L^k_{\P}}\leq C\sqrt{k}\|\e^{-i(t+\tau)H}H^{\s/2}\gamma^{\om}(x)\|_{L^2_{\P}}= 
 C\sqrt{k}\big(\sum_{n=0}^{+\infty}\lambda^{\s}_{n}|c_{n}|^{2}|\phi_{n}(x)|^2\big)^{1/2}.
 \end{equation*}
Assume that   $k\geq r$, then by the integral Minkowski inequality, the previous line and the triangle inequality, we get
\begin{eqnarray} 
\|\e^{-i(t+\tau)H}H^{\s/2}\gamma^{\om}\|_{L^k_{\P}L^r_x}&\leq& \|\e^{-i(t+\tau)H}H^{\s/2}\gamma^{\om}\|_{L^r_xL^k_{\P}}\nonumber\\
&\leq &  C\sqrt{k}\big\|\sum_{k=0}^{+\infty}\lambda^{\s}_{k}|c_{k}|^{2}|\phi_{k}|^2\big\|^{1/2}_{L^{r/2}(\R^{d})}\nonumber\\
&\leq &  C\sqrt{k}\Big(\sum_{j=1}^{+\infty} \big\|\sum_{k\in I(j)}\lambda^{\s}_{k}|c_{k}|^{2}|\phi_{k}|^2\big\|_{L^{r/2}(\R^{d})}\Big)^{1/2}.\label{pom}
\end{eqnarray}
Condition \eqref{condi0} implies  that for all $x\in \R^{d}$ and $k\in I(j)=\big\{n\in \N,\;2j\leq \lambda_{n}< 2(j+1)\big\}$
 \begin{equation*}
 \lambda^{\s}_{k}|c_{k}|^{2}|\phi_{k}(x)|^2\leq Cj^{\s}\sum_{n\in I(j)}|c_{n}|^{2} \cdot \frac{|\phi_{k}(x)|^{2}}{\# I(j)},
 \end{equation*}
 and thus, by Lemma \ref{lemme5}, and the fact that $\# I(j) \sim c j^{d-1}$
 \begin{eqnarray*}
 \big\|\sum_{k\in I(j)} \lambda^{\s}_{k}|c_{k}|^{2}|\phi_{k}(x)|^2\big\|_{L^{r/2}(\R^{d})}&\leq &Cj^{\s}\sum_{n\in I(j)}|c_{n}|^{2} \cdot \frac{ \big\|    \sum_{k\in I(j)}|\phi_{k}(x)|^{2}   \big\|_{L^{r/2}(\R^{d})}     }{\# I(j)}\\
 &\leq & C j^{\s+d(1/r-1/2)}\sum_{n\in I(j)}|c_{n}|^{2}\\
 &= & C j^{s}\sum_{n\in I(j)}|c_{n}|^{2}.
 \end{eqnarray*}
 The latter inequality together with \eqref{pom} gives 
 \begin{equation*}
 \|\e^{-i(t+\tau)H}H^{\s/2}\gamma^{\om}\|_{L^k_{\P}L^r_x} \leq C\sqrt{k}\|\gamma\|_{\H^{s}(\R^{d})},
 \end{equation*}
 and for $k\geq r$, by Minkowski, 
  \begin{equation*}
 \|\e^{-i(t+\tau)H}H^{\s/2}\gamma^{\om}\|_{L^k_{\P}L^{q}_{[0,T]}L^r_x} \leq C\sqrt{k}\,T^{1/q}\|\gamma\|_{\H^{s}(\R^{d})}.
 \end{equation*}
  Then, using \eqref{mu.defi}and the  Bienaym\'e-Tchebichev inequality, we obtain
\begin{eqnarray*}
\mu\big(\,u\in \H^{s}\,:\, \|\e^{-i(t+\tau)H}u\|_{L^{q}_{[0,T]}\W^{\s,r}(\R^{d})}>K\,\big)&\leq& 
(K^{-1}\|\e^{-i(t+\tau)H}H^{\s/2}\gamma^{\om}\|_{L^k_{\P}L^{q}_{[0,T]}L^r_x})^k\\
&\leq&
(CK^{-1}\sqrt{k}\,T^{1/q}\|\gamma\|_{\H^{s}(\R^{d})} )^{k}\,.
\end{eqnarray*}
Finally, if $ K \gg \|\gamma\|_{{\mathcal H}^{s}(\R^{d})} $,  we can choose $\dis k=     \frac{ K^{2}}{   2CT^{2/q}\|\gamma\|^{2}_{\H^{s}(\R^{d})} }     \geq  r$, which yields the result.

$\bullet$ Assume now that  $r=+\infty$. We use the Sobolev inequality to get $\|u\|_{\W^{s,\infty}}\leq C\|u\|_{\W^{\tilde{s},\tilde{r}}}$ with $\tilde{s}=s+2d/\tilde{r}$ for $\tilde{r}\geq 1$ large enough; hence we can apply the previous result for $r<+\infty$.

         \begin{rema}
         A similar result to Proposition \ref{regularite} holds, with the same gain of derivatives, when $I(\lambda)$ is replaced with the dyadic interval $J(j)=\big\{n\in \N\;,2^{j} \leq \lambda_{n}< 2^{j+1}\big\}$. Then  the condition \eqref{condi0} becomes 
    \begin{equation}\label{condi2} 
  |c_{k}|^{2}\leq \frac{C}{\# J(j)}\sum_{n\in J(j)}|c_{n}|^{2},\quad \forall k\in J(j),\quad \forall  j\geq 1,
  \end{equation}
          which seems more  restrictive. Indeed none of the conditions imply the other. 
          
          Observe that if we want to prove the result under condition \eqref{condi2},  the subtle estimate \eqref{kara3} is not needed, \eqref{karest} is enough.
  \end{rema}
          
  \begin{rema}
  For $d=1$ condition (\ref{condi0}) is always satisfied but  condition (\ref{kara3}) is not. Instead  we can use that $\Vert \phi_{k} \Vert_p \leq C\lambda_{k}^{-\theta(p)}$ with  $\theta(p) >0$ for $p>2$ (see \cite{KOTA}). For example if $p>4$ we have $\theta(p) = \frac{1}{4} -\frac{p-1}{6p}$. So we get  the  Proposition \ref{regularite}    with $s=p\theta(p)/4$ (see \cite{tho,BTT} where this is used).             
  \end{rema}

       \begin{rema}
Another approach could have been to exploit the particular basis $(\phi_{n})_{n\geq 1}$ which satisfy the good $L^{\infty}$ estimates given in \cite[Theorem 1.3]{PRT1}, and to construct  the measures $\mu$  as the image measures of  random series   of the form 
$$\gamma^{\om}(x)=\sum_{n\geq 1} c_{n}g_{n}(\om)\phi_{n}(x),$$     
with $c_{n}\in \ell^{2}(\N)$ which does not necessarily satisfy \eqref{condi0}. A direct application of the Khinchin inequality (as {\it e.g.} in \cite[Proposition 2.3]{tho}) then gives the same bounds as in Proposition \ref{regularite}. Observe that the condition   \eqref{condi0} is also needed in this approach, but it directly intervenes  in the construction of the $\phi_{n}$.

We believe that the strategy we adopted here is  slightly more general, since it seems to work even in  cases where we do not have a basis of eigenfunctions which satisfy analogous  bounds to \cite[Theorem~1.3]{PRT1}, as for example in the case of the operator $-\Delta+|x|^{4}$.
  \end{rema}

     
\section{Application to the local theory of the super-critical Schr\"odinger   equation}\label{Sect3}

\subsection{Almost sure local well-posedness}
 This subsection, devoted to the proof of Theorem \ref{letheoreme}, follows the argument of \cite{poiret2}. 
 
 Let $u_{0}\in L^{2}(\R^{d})$. We look for a solution to \eqref{SH} of the form $ u = \e^{-itH} u_0 + v $, where $v$ is some fluctuation term which is more regular than the linear profile $\e^{-it H}u_{0}$. By the Duhamel formula, the unknown $v$ has to be a fixed point of the operator
 \begin{equation}\label{def.L}
L ( v): =    \mp i \int_0^t   \e^{-i(t-s)H} |   \e^{-isH} u_0 + v(s) |^{p-1} (  \e^{-isH} u_0 + v(s) )  \, \text{d}s,
\end{equation}
 in some adequate functional space, which is a Strichartz space.\ligne

To begin with, we recall the Strichartz  estimates for the harmonic oscillator.  A couple $(q,r)\in [2,+\infty]^2$ is called admissible if 
\begin{equation*}
\frac2q+\frac{d}{r}=\frac{d}2\quad \text{and}\quad (d,q,r)\neq (2,2,+\infty),
\end{equation*}
and if one defines 
\begin{equation*}
 {X}^s_{T}:= \bigcap_{(q,r) \ admissible} L ^q\big( [-T, T] \,; \W^{s,r}( \R^d )\big),
\end{equation*}
then for all $T>0$ there exists $C_{T}>0$ so that for all $u_{0}\in \H^{s}(\R^{d})$ we have 
\begin{equation}\label{Stri0}
\|\e^{-itH}u_{0}\|_{X^{s}_{T}}\leq C_{T}\|u_{0}\|_{\H^{s}(\R^{d}).}
\end{equation}
We will also need the inhomogeneous version of Strichartz:  For all $ T>0 $, there exists $C_{T}>0$ so that for all   admissible couple       $ ( q, r ) $ and function  $ F \in L^{q'}( [T,T]; \W^{s,r'} (\R^d)) $,
\begin{equation}\label{Stri}
 \big\|  \int _0^t \e^{-i(t-s)H} F(s) ds   \big\| _{   X^s_T} \leq C_{T} \| F \|_{  L^{q'} ([-T,T],{\W}^{s,r'} (\R^d)) },
\end{equation}
where $ q' $ and $ r'$ are the H\"older conjugate of $ q $ and $ r $. We refer to \cite{poiret2} for a proof.\ligne

The next result is a direct application of the Sobolev embeddings and H\"older.
\begin{lemm}\ph \label{injection}
Let $ (q,r) \in [2,\infty [ \times [2,\infty ]$, and let  $s,s_0 \geq 0 $ be so that $ s-s_0 >  \frac{d}{2} - \frac{2}{q} - \frac{d}{r} $. Then there exist $ \kappa,C > 0 $ such that for any $ T \geq 0 $ and $  u \in {X}_T^s $,
\begin{equation*}
\| u \|_{ L^q( [-T,T] , \W^{s_0,r}(\R^d)) }  \leq  C T ^\kappa \|u\|_{  {X}_T^s }.
\end{equation*}
\end{lemm}
 
We now introduce the appropriate sets in which we can take profit of the stochastic estimates of the previous section. Fix $\mu=\mu_{\gamma}\in \M^{0}$ and for  $ K   \geq 0 $ and $\eps> 0 $,   define the set $ G_{d}(K) $ as 
 \begin{equation*}
G_{d}(K)=\big\{ w\in L^{2}(\R^{d})\,:\;\;   \|w\|_{  L^2(\R^d) }   \leq K    \;\;\text{and}\;\;   \|  \e^{-itH} w\|_{ L_{[-2\pi,2\pi]}^{1/\eps}\W^{  d/2-\eps , \infty } (\R^d) }\leq K \; \big\}.
\end{equation*}
 Then  by Proposition \ref{regularite}, 
 \begin{equation}\label{G.D}
 \mu\big( \,(G_{d}(K))^{c}\,\big)\leq  \mu\big( \, \|w\|_{  L^2(\R^d) }   > K\,  \big)+ \mu\big( \, \|  \e^{-itH} w \|_{ L_{[-2\pi,2\pi]}^{1/\eps}\W^{  d/2-\eps , \infty } (\R^d) }>   K\,  \big)\leq C\e^{-\frac{cK^{2}}{\|\gamma\|^{2}_{L^{2}}}}.
 \end{equation}
 
 We want to perform a fixed point argument on $L$ with initial condition $u_{0}\in G_{d}(K)$ for some $K>0$ and $\eps>0$ small enough. We begin by establishing some estimates.
\begin{lemm}\ph \label{estim1}
Let $ s \in  \big]   \frac{d}{2}   - \frac{2}{p-1}  ; \frac{d}{2}     \big[ $ then for $\eps>0$ small enough there exist $ C >  0 $ and $ \kappa > 0 $ such that  for any $ 0 < T \leq 1 $, $u_{0}\in G_{d}(K)$, $ v \in {X}^s_T $ and $ f_i = v $ or $ f_i = \e^{-itH} u_0 $,
\begin{equation}\label{esti1}
\big\| \,H^{s/2}( v  )    \prod_{ i = 2 } ^p  f_i  \big\|_{  L^1( [-T,T]  , L ^2(\R^d)) } \leq C  T ^\kappa(  K ^p + \|v\|^p_{{X}^s_T}  ),
\end{equation}
and
\begin{equation}\label{esti2}
\big\| \,H^{s/2}( \e^{-itH} u_0   )    \prod_{ i = 2 } ^p  f_i  \big\|_{  L^1( [-T,T]  , L ^2(\R^d)) } \leq C  T ^\kappa(  K ^p + \|v\|^p_{{X}^s_T}  ).
\end{equation}
\end{lemm}

\begin{proof}
$\bullet$ First we prove \eqref{esti1}. Thanks to the H\"older inequality,
\begin{align*}
\big\| \,|\nabla|^s( v  )    \prod_{ i = 2 } ^p  f_i  \big\|_{  L^1( [-T,T]  , L ^2(\R^d)) } & \leq \| \,| \nabla|^s( v  ) \|_{  L^\infty ( [-T,T], L^2(\R^d)) }  \prod_{ i = 2 } ^p \| f_i \|_{  L^{p-1} ( [-T,T], L^\infty(\R^d)) }
\end{align*}
and
\begin{align*}
\big\| \<x\>^s   v      \prod_{ i = 2 } ^p  f_i  \big\|_{  L^1( [-T,T]  , L ^2(\R^d)) } & \leq \| \<x\>^s v   \|_{  L^\infty ( [-T,T], L^2(\R^d)) }  \prod_{ i = 2 } ^p \| f_i \|_{  L^{p-1} ( [-T,T], L^\infty(\R^d)) }
\\ & \leq \| v \|_{  L^\infty ( [-T,T], {H}^s (\R^d)) }  \prod_{ i = 2 } ^p \| f_i \|_{  L^{p-1} ( [-T,T], L^\infty(\R^d)) }.
\end{align*}
If $ f_i = v $, then as $ s > \frac{d}{2} - \frac{2}{p-1} $, we can use  Lemma \ref{injection} to obtain \begin{equation*}
\| v \|_{  L^{p-1} ( [-T,T], L^\infty(\R^d)) }  \leq C T^{\kappa  } \| v \|_{  {X}^s_T }.
\end{equation*}
If $ f_i = \e^{-itH} u_0 $,  then by definition of $G_{d}(K)$ we have for $\eps>0$ small enough
\begin{equation*}
\| \e^{-itH} u_0  \|_{  L^{p-1} ( [-T,T], L^\infty(\R^d)) }  \leq T^{   \kappa } \| \e^{-itH} u_0 \|_{ L^{1/\eps} (  [-2 \pi, 2 \pi ] ,  \W^{d/2-\eps,\infty} ( \R^d)  )  }  \leq T^{   \kappa }  K. 
\end{equation*}
$\bullet$ We now turn to \eqref{esti2}. Thanks to the H\"older inequality, we have 
\begin{multline*}
 \quad \big\| \,|\nabla|^s( \e^{-itH} u_0   )    \prod_{ i = 2 } ^p  f_i  \big\|_{  L^1( [-T,T]  , L ^2(\R^d)) }  \leq\\ 
\begin{aligned}
&\leq \| \  |\nabla|^s( \e^{-itH} u_0  )   | |_{  L^{p } (  [-T, T ] ,  L ^ {  2dp }  ( \R^d )  )  }  \prod_{ i = 2 } ^p \|  f_i \|_{  L^{p} ( [-T,T],  L^{ \frac{2dp(p-1)}{dp-1} }(\R^d)) }
\\ 
&  \leq \  \|   \e^{-itH} u_0     | |_{  L^{p } (  [-T, T ] ,  \W ^ { s, 2dp }  ( \R^d )  )  }  \prod_{ i = 2 } ^p \|  f_i \|_{  L^{p} ( [-T,T],  L^{ \frac{2dp(p-1)}{dp-1} }(\R^d)) }.
\end{aligned}
\end{multline*}
If $ f_i = \e^{-itH} u_0 $, by interpolation, we obtain for some $0\leq \theta\leq 1$
\begin{align*}
 \| \e^{-itH} u_0  \|_{   L^{p} ( [-T,T],  L^{ \frac{2dp(p-1)}{dp-1} }(\R^d))}   & \leq  C  T ^{   \kappa  }   \|  u_0  \|^{1-\theta}_{   L^2(\R^d)  }  \| \e^{-itH} u_0  \|^{\theta}_{   L^{1/\eps} ( [-T,T],   L^{ \infty  }(\R^d))  }
 \\ & \leq C  T ^{   \kappa  }  K.
\end{align*}
If $ f_i = v $, as $ s > \frac{d}{2} - \frac{2}{p-1}  > \frac{d}{2} - \frac{2}{p} - \frac{d(dp-1)}{2dp(p-1)} $ (because $ p \geq 3 $ and $ d \geq 2 $) then thanks to  Lemma \ref{injection}, we find 
\begin{equation*}
\| v  \|_{  L^{p} ( [-T,T],  L^{ \frac{2dp(p-1)}{dp-1} }(\R^d)) }  \leq C  T^\kappa  \|v\|_{{X}^s_T } . 
\end{equation*}
\end{proof}

 We are now able to establish the estimates which will be useful to apply a fixed point theorem.
\begin{prop}\ph \label{pointfixe1} Let $ s \in \big] \frac{d}{2}   - \frac{2}{p-1}  ; \frac{d}{2}    \big[ $. Then for $\eps>0$ small enough,  there exist $ C > 0 $ and $ \kappa > 0 $ such that if $ u_0 \in G_{d}(K)$ for one $ K > 0 $ then for any $ v,v_{1},v_{2} \in {X}_T^s $ and $ 0 < T \leq 1 $,
 \begin{equation*}
\bigg\|  \int_0^t   \e^{-i(t-s)H}   |  \e^{-isH} u_0 + v |^{p-1}   (   \e^{-isH} u_0 + v ) \, \text{d}s   \bigg\|_{  {X}_T^s }  
 \leq C  T^\kappa   ( K ^p +  \|v\|^p_{  {X}_T^s } ),
\end{equation*} 
and
 \begin{align*}
& \bigg\|  \int_0^t   \e^{-i(t-s)H}    |  \e^{-isH} u_0 + v_1 |^{p-1}   (  \e^{-isH} u_0 + v_1 )  \, \text{d}s  
\\ & \hspace*{3cm} -  \int_0^t   \e^{-i(t-s)H}  |  \e^{-isH} u_0 + v_2 |^{p-1}   (  \e^{-isH} u_0 + v_2 )  \, \text{d}s  \bigg\|_{  {X}_T^s }  \leq 
\\ & \leq C  T ^\kappa   \|v_1-v_2\|_{ {X}^s_T }  ( K ^{p-1} +  \|v_1\|^{p-1}_{  {X}_T^s } + \|v_2\|^{p-1}_{  {X}_T^s } ).
\end{align*} 
\end{prop} 
\begin{proof}
We only prove the first claim, since the proof of the second is similar. Using the Strichartz inequalities \eqref{Stri}, we obtain
\begin{multline*}
 \bigg\|  \int_0^t   \e^{-i(t-s)H}  |  \e^{-isH} u_0 + v |^{p-1}  (   \e^{-isH} u_0 + v ) \ \text{d}s   \bigg\|_{  {X}_T^s }  \leq \\
 \begin{aligned}
&
 \leq  C  \big\|  \,  |  \e^{-isH} u_0 + v |^{p-1}  (   \e^{-isH} u_0 + v )   \big\|_{   L_{ [-T,T] }^1 {\H}^s(\R^d)   }.
\end{aligned}
\end{multline*}

Then, using  Lemma \ref{estim1}, we obtain the existence of $ \kappa > 0  $ such that for any $ u_0 \in G_{d}(K) $, $ 0 < T \leq 1 $ and $ v  \in {X}^s_T $,
\begin{align*}
\big\| \,  H^{s/2}  \left(  |  \e^{-isH} u_0 + v |^{p-1}  (   \e^{-isH} u_0 + v )  \right)   \big\|_{   L^1( [-T,T] , L^2 (\R^d) )   }  \leq C T ^\kappa ( K^p +   \|v\|^p_{  {X}_T^s }   ).
\end{align*}
\end{proof}
\begin{proof}[Proof of Theorem \ref{letheoreme}]
We   now  complete the contraction argument on $L$ defined in \eqref{def.L} with some   $u_{0} \in G_{d}(K)$. 
 According to the Proposition \ref{pointfixe1}, there exist $ C > 0 $ and $ \kappa > 0 $ such that 
\begin{align*}
& \| L (v)  \|_{  { X } ^s_T  } \leq C T^\kappa ( K ^p + \| v\|^p _{ { X } ^s_T }) 
\\ & \| L (v_1)-L(v_2)  \|_{  { X } ^s_T  } \leq C T^\kappa \| v_1-v_2\| _{ { X } ^s_T } ( K ^{p-1} + \| v_1\|^{p-1}_{ { X } ^s_T }+ \| v_2\|^{p-1}_{ { X } ^s_T } ).
\end{align*}
Hence, if we choose $T>0$ such that  $ K =   (  \frac{1}{8 C  T^\kappa }  )^{  \frac{1}{p-1} }    $ then $ L $ is a contraction in the space $ B_{{X}^{s}_T } (0, K   ) $ (the ball of radius $K$ in $X^{s}_{T}$). Thus if we set    $\Sigma_{T}=G_{d}(K)$, with the previous choice of $K$, the result follows from \eqref{G.D}.
\end{proof}

\begin{proof}[Proof of Theorem \ref{Thm12}]
We introduce 
\begin{equation} \label{NLSH}  
  \left\{
      \begin{aligned}
         & i \frac{ \partial w }{ \partial t } - H w = \pm \cos(2t)^{\frac{d}{2}(p-1)-2}  |w|^{p-1} w, \quad (t,x)\in \R\times \R^{d},
       \\  &  v(0)  =u_{0},
      \end{aligned}
    \right.
\end{equation}
and let $ s \in \big] \frac{d}{2}   - \frac{2}{p-1}  ; \frac{d}{2}    \big[ $, $ T = \frac{\pi}{4} $ and $ 1 \gg \epsilon > 0 $. Thanks to Proposition \ref{pointfixe1}, there exist $ C > 0 $ and $ \kappa > 0 $ such that if $ u_0 \in G_{d}(K)$ for one $ K > 0 $ then for all $ v $,
 \begin{equation} \label{limite}
\bigg\|  \int_0^t  \e^{-i(t-s)H}\Big[ \cos(2s)^{\frac{d}{2}(p-1)-2}      |   \e^{-isH} u_0 + v |^{p-1}   (   \e^{-isH} u_0 + v ) \, \Big]\text{d}s   \bigg\|_{  {X}_T^s }  
 \leq C  T^\kappa   ( K ^p +  \|v\|^p_{  {X}_T^s } ).
\end{equation} 
As in Theorem \ref{letheoreme}, we can choose $  K =   (  \frac{1}{8 C  T^\kappa }  )^{  \frac{1}{p-1} }  $ to obtain, for $ u_0 \in G_{d}(K)$, a unique local solution $ w = \e^{-itH} u_ 0 + v  $ in time interval $ \left] - \frac{\pi}{4} , \frac{\pi}{4} \right[ $ to (\ref{NLSH}) with $ v \in X^1_T $.
\\We set $ u =  \mathscr{L} v  $, then $ u $ is a global solution to (\ref{NLS}). Thanks to  \cite[Propositions 20 and 22]{poiret2}, we obtain that $ u = \e^{it\Delta} u_ 0 + v' $ with $ v' \in X^1_T $.
\\Moreover, thanks to \eqref{limite}, we have that 
\begin{equation*}
 \int_0^t \e^{-i(t-s)H}\Big[ \cos(2s)^{\frac{d}{2}(p-1)-2}      |   \e^{-isH} u_0 + v |^{p-1}   (   \e^{-isH} u_0 + v ) \, \Big]\text{d}s   \in \mathcal{C}^0\big( [-T,T], \H^s(\R^{d}) \big).
\end{equation*} 
Then, there exist $ L \in \H^s $ such that
\begin{equation*}
\underset{ t \rightarrow T }{\lim} \Big\| e^{-itH } \int_0^t  \e^{-isH}\Big[ \cos(2s)^{\frac{d}{2}(p-1)-2}      |   \e^{-isH} u_0 + v |^{p-1}   (   \e^{-isH} u_0 + v ) \, \Big] \text{d}s  -L \Big\|_{\H^s(\R^{d})} =0 .
\end{equation*}
Using  \cite[Lemma 70]{poiret2}, we obtain that 
\begin{equation*}
\underset{ t \rightarrow T }{\lim} \Big\| u(t) - \e^{it\Delta}u_0  - \e^{it\Delta} \left( -i \e^{-iTH} L \right) \Big\|_{H^s(\R^{d})} =0 .
\end{equation*}

Finally to establish Theorem \ref{Thm12}, it suffices to set $ \Sigma = G_{d}(K) $ and to prove that $ \mu \big( u_0 \in G_{d}(K)  \big) > 0 $. We can write
\begin{equation*}
u_0 = \chi \Big( \frac{H}{N} \Big) u_0 + (1- \chi ) \Big( \frac{H}{N} \Big) u_0 := [u_0]_N+[u_0]^N,
\end{equation*}
with $ \chi $ is a truncation function. Using the triangular inequality and the independence, we obtain that
\begin{equation*}
\mu\big( u_0 \in G_{d}(K)   \big) \geq \mu \big(  [u_0]_N \in G_{d}( \frac{K}{2}  )   \big)  \,\mu  \big( [u_0]^N \in G_{d}( \frac{K}{2}  )  \big) .
\end{equation*}
For all $ N $, $ \mu \left(  [u_0]_N \in G_{d}( \frac{K}{2}  )   \right) > 0 $ because the hypothesis (\ref{petit}) is satisfied and thanks to Proposition~\ref{regularite}, we have
\begin{equation*}
 \mu \big(  [u_0]^N \in G_{d}( \frac{K}{2}  )   \big) \geq 1-  C \e^{   -  \frac{c K^2}{ \| [u_0]^N \|^2_{L^2}}}   \underset{N \rightarrow \infty}{\longrightarrow} 1
\end{equation*}
and there exists $ N $ such that $  \mu \big(  [u_0]^N \in G_{d}( \frac{K}{2}  )   \big)  > 0 $.
\end{proof}

 \subsection{Almost sure local well-posedness of the time dependent equation and scattering for NLS}\label{sect32}

 This section is devoted to the proof of Theorem \ref{Thm14}. The strategy is similar to the proof of Theorem \ref{Thm12}: we solve the equation which is mapped by $\mathscr{L}$ to \eqref{NLS} up to time $T=\pi/4$ and we conclude as previouly. The difference here, is that the nonlinear term of the equation we have to solve is singular a time $T=\pi/4$. More precisely, we consider the equation
\begin{equation} \label{NLST}  
  \left\{
      \begin{aligned}
         &  \dis i \frac{ \partial u }{ \partial t } -H u  =   \pm \cos(2t)^{  p-3 } | u|^{p-1} u, \quad (t,x)\in \R\times \R^{2},
       \\  &  u(0)  =u_0,
      \end{aligned}
    \right.
\end{equation}
when $2<p<3$. 

\medskip
Let us first consider the easier case $\s>0$.

 \subsubsection{{\bf Proof of Theorem \ref{Thm14} in the case $\boldsymbol {\s>0}$:}}
 
Let  $\s>0$ and $\mu=\mu_{\gamma}\in \M^{\s}$ and for  $ K   \geq 0 $     and $\eps>0$, define the set $ F_{\s}( K ) $ as 
 \begin{equation*}
F_{\s}(K)=\big\{ w\in \H^{\s}(\R^{2})\,:\;\;   \|w\|_{  \H^{\s}(\R^2) }   \leq K    \;\;\text{and}\;\;   \|  \e^{-itH} w\|_{ L_{[0,2\pi] }^{1/\eps} \W^{  1+\s-\eps ,\infty   } (\R^2)  }\leq K \; \big\}.
\end{equation*}
The parameter $\eps>0$ will be chosen small enough so that we can apply  Proposition \ref{regularite} and get 
 \begin{equation*} 
 \mu\big( \,(F_{\s}(K))^{c}\,\big)\leq  \mu\big( \, \|w\|_{  \H^{\s} }   > K\,  \big)+ \mu\big( \, \|  \e^{-itH} w \|_{ L^{1/\eps}_{ [0,2\pi]} \W^{1+  \s-\eps , \infty }   }>   K\,  \big)\leq C\e^{-\frac{cK^{2}}{\|\gamma\|^{2}_{\H^{\s}}}}.
 \end{equation*}

 The next proposition is the key   in the proof of Theorem~\ref{Thm14} when $\s>0$.

\begin{prop}\ph \label{pointfixe2} Let $\s>0$. There exist $ C > 0 $ and $ \kappa > 0 $ such that if $ u_0 \in F_{\s}(K)$ for one $ K > 0 $ then for any $ v,v_{1},v_{2} \in {X}_T^1 $ and $ 0 < T \leq 1 $,
 \begin{equation}\label{f1}
\bigg\|  \int_0^t  \e^{-i(t-\tau)H}\Big[   \cos(2\tau)^{  p-3 } |  \e^{-i\tau H} u_0 + v |^{p-1}   (   \e^{-i\tau H} u_0 + v ) \Big]\, \text{d}\tau   \bigg\|_{  {X}_T^1 }  
 \leq C  T^\kappa   ( K ^p +  \|v\|^p_{  {X}_T^1 } ),
\end{equation} 
and
 \begin{align}\label{f2}
& \bigg\|  \int_0^t   \e^{-i(t-\tau)H}  \Big[\cos(2\tau)^{  p-3 }   |  \e^{-i\tau H} u_0 + v_1 |^{p-1}   (  \e^{-i\tau H} u_0 + v_1 )\Big]  \, \text{d}\tau  
 \\ & \hspace*{3cm} -  \int_0^t    \e^{-i(t-\tau)H}\Big[  \cos(2\tau)^{  p-3 } |  \e^{-i\tau H} u_0 + v_2 |^{p-1}   (  \e^{-i\tau H} u_0 + v_2 ) \Big] \, \text{d}\tau  \bigg\|_{  {X}_T^1 }  \nonumber\leq
\nonumber\\ & \leq C  T ^\kappa   \|v_1-v_2\|_{ {X}^1_T }  ( K ^{p-1} +  \|v_1\|^{p-1}_{  {X}_T^1 } + \|v_2\|^{p-1}_{  {X}_T^1 } ).\nonumber
\end{align} 
\end{prop} 
\begin{proof}
We first prove \eqref{f1}. Using the Strichartz inequalities \eqref{Stri}, we obtain
\begin{multline*}
 \bigg\|  \int_0^t   \e^{-i(t-\tau)H}  \Big[\cos(2\tau)^{  p-3 }  |  \e^{-i\tau H} u_0 + v |^{p-1}  (   \e^{-i\tau H} u_0 + v )  \Big]\ \text{d}\tau   \bigg\|_{  {X}_T^1 }  \leq \\
 \begin{aligned}
& \leq  C   \big\|  \cos(2\tau)^{  p-3 }   |  \e^{-i\tau H} u_0 + v |^{p-1}  (   \e^{-i\tau H} u_0 + v )    \big\|_{   L_{[-T,T] }^1 \H^1 (\R^2)  } . 
\end{aligned}
\end{multline*}
We use the formula 
\begin{equation}\label{for}
  \nabla (|u|^{p-1}u)=\frac{p+1}2 |u|^{p-1}\nabla u+\frac{p-1}2 |u|^{p-3}u^{2}\nabla \ov{u}.
\end{equation} 
We denote by $f=\e^{-isH}u_{0}$
\begin{multline*}
 \big\|  \, \nabla  \big(  |  f + v |^{p-1}  (   f + v ) \big)  \, \big\|_{    L^2 (\R^2)   }    \leq \\
 \begin{aligned}
& \leq  C  \big\| \, |f+v|^{p-1} \nabla     (   f + v ) \,  \big\|_{    L^2 (\R^2)   } +C  \big\| \, |f+v|^{p-3}(f+v)^{2} \nabla     (   f + v ) \,  \big\|_{    L^2 (\R^2)    } \\
& \leq  C  \big\| f+v   \big\|^{p-1}_{L^\infty (\R^2)}\big\|  \nabla    v  \big\|_{L^2 (\R^2)}  +C  \big\| f+v   \big\|^{p-1}_{L^{2(p-1)}(\R^2)}\big\|  \nabla    f  \big\|_{L^{\infty} (\R^2)}.
\end{aligned}
\end{multline*}
Therefore
\begin{multline*} 
 \big\|   |  f + v |^{p-1}  (   f + v )    \big\|_{    \H^1 (\R^2)   }    \leq \\
 \leq C   (\big\| f  \big\|^{p-1}_{L^\infty (\R^2)} +\big\| v  \big\|^{p-1}_{L^\infty (\R^2)})\big\|      v  \big\|_{\H^1 (\R^2)}  +C ( \big\| f \big\|^{p-1}_{L^{2(p-1)} (\R^2)}+\big\| v  \big\|^{p-1}_{L^{2(p-1)}(\R^2)})\big\|      f  \big\|_{\W^{1,\infty} (\R^2)}.
\end{multline*}

Now observe that $\|v\|_{L^{\infty}_{[-T,T]}L^{2(p-1)}}\leq \|v\|_{X^{1}_{T}}$ as well as for all $r<+\infty$, $\|v\|_{L^{r}_{[-T,T]}L^{\infty}}\leq \|v\|_{X^{1}_{T}}$. Then for all $q>1$
\begin{multline}\label{thr}
 \big\|   |  f + v |^{p-1}  (   f + v )    \big\|_{ L^{q}_{[-T,T]}   \H^1 (\R^2)   }    \leq \\
  \begin{aligned}
& \leq C  T^{\kappa} \Big[ \big(\big\| f  \big\|^{p-1}_{L^{\infty-}_{[-T,T]}L^\infty (\R^2)} +\big\| v  \big\|^{p-1}_{X^{1}_{T}}\big)\big\|      v  \big\|_{X^{1}_{T}}  +  \big(\big\| f \big\|^{p-1}_{L^{\infty-}_{[-T,T]}L^{2(p-1)} (\R^2)}+\big\| v  \big\|^{p-1}_{X^{1}_{T}}\big)\big\|      f  \big\|_{L^{\infty-}_{[-T,T]}\W^{1,\infty} (\R^2)}\Big]\\
&\leq C  T^\kappa   ( K ^p +  \|v\|^p_{  {X}_T^1 } ).
\end{aligned}
\end{multline}
Choose  $q>1$ so that $q'(3-p)<1$, then we have $\|  \cos(2\tau )^{  p-3 }\|_{L^{q'}_{[-T,T]}} <\infty$, thus from \eqref{thr} and H\"older, we infer
\begin{multline*}
 \big\|  \cos(2\tau)^{  p-3 }  \,    |  f + v |^{p-1}  (   f + v )   \big\|_{   L^1( [-T,T] , \H^1 (\R^2) )   }    \leq \\
 \begin{aligned}
& \leq  C   \|  \cos(2\tau)^{  p-3 }\|_{L^{q'}_{[-T,T]}}    \big\|\,    |  f + v |^{p-1}  (   f + v )   \big\|_{   L_{ [-T,T] }^{q} \H^1 (\R^2)    }  \\
& \leq C  T^\kappa   ( K ^p +  \|v\|^p_{  {X}_T^1 } ).
\end{aligned}
\end{multline*}

For the proof of \eqref{f2} we can proceed similarly. Namely, we use the estimates

\begin{equation}\label{diff}
 \big|\,  |z_1|^{p-1} - |z_2|^{p-1}  \big|\le
C ( |z_1|^{p-2} +  |z_2|^{p-2})  |z_1-z_2|
\end{equation}
and
\begin{equation*}  
 \big|\,  |z_1|^{p-3} z_1^2 -  |z_2|^{p-3} z_2^2  \big|\le
   C ( |z_1|^{p-2} +  |z_2|^{p-2})   |z_1-z_2| , 
\end{equation*}
which are proven in \cite[Remark 2.3]{CFH} together with \eqref{for}.
 \end{proof}

 \subsubsection{{\bf Proof of Theorem \ref{Thm14} in the case ${\boldsymbol{ \s   =0}}$:}}
The strategy of the proof in this case is similar, at the price of some technicalities, since the Leibniz rule \eqref{for} does not hold true for non integer derivatives. Actually, when $\s=0$, we will have to work in $X^{s}_{T}$ for  $s<1$ because  the probabilistic term $\e^{-itH}u_{0}\notin \mathcal{W}^{1,\infty}(\R^{2})$.
 
Moreover, we are not able to obtain a contraction estimate in $X^{s}_{T}$. Therefore, we will do a fixed point in the   space $\big\{ \|v\|_{X^{s}_{T}}\leq K    \big\}$ endowed with the weaker metric induced by $X^{0}_{T}$. We can check that this space is complete. Actually,  by the Banach-Alaoglu theorem, the closed balls of each  component spaces of  $X^{s}_{T}$ is compact for  the $\text{weak}^{\star}$ topology.\medskip

For $0<s<1$, we use the following characterization of the usual $H^{s}(\R^{2})$ norm
\begin{equation}\label{diffsobo}
\|g\|_{H^{s}(\R^{2})}=\Big(   \int_{\R^{2}\times \R^{2}}\frac{|g(x)-g(y)|^{2}}{|x-y|^{2s+2}}\text{d}x\text{d}y   \Big)^{1/2}.
\end{equation}

For  $\mu=\mu_{\gamma}\in \M^{0}$,  $ K   \geq 0 $     and $\eps>0$, define the set $ \wt{F}_{0}( K ) $ as 
 \begin{multline*}
 \wt{F}_{0}( K ) =\big\{ w\in L^{2}(\R^{2})\,:\;\;   \|w\|_{  L^{2}(\R^2) }   \leq K,   \;  \|  \e^{-itH} w\|_{ L_{[0,2\pi] }^{1/\eps} {\W}^{  1-\eps ,\infty  } (\R^2)  }\leq K   \;\;\\\text{and}\;\;   \big \|  (\e^{-itH} w)(x)-(\e^{-itH} w)(y)\big\|_{ L_{t\in [0,2\pi] }^{\infty}   }\leq K|x-y|^{1-\eps} \; \big\}.
\end{multline*}
The next results states that $\wt{F}_{0}(K)$ is a  set with large measure.
 \begin{lemm}\ph\label{lemf0} If $\eps>0$ is small enough
\begin{equation*} 
 \mu\big( \, (\wt{F}_{0}( K ))^{c}\,\big)\leq  C\e^{-\frac{cK^{2}}{\|\gamma\|^{2}_{L^{2}(\R^{2})}}}.
 \end{equation*} 
  \end{lemm}

\begin{proof}
We only have to study the contribution of the Lipschitz term in $ \wt{F}_{0}( K ) $, since the others are controlled by Proposition \ref{regularite}.

We fix $\dis\gamma=\sum_{n=0}^{+\infty}c_{n} \phi_{n}\in \A_{0}$ and denote by $\dis \gamma^{\om}=\sum_{n=0}^{+\infty}g_{n}(\omega)c_{n}\phi_{n}$.  Let $k\geq 1$, then  by definition 
\begin{equation}\label{mu.defi2}
\int_{L^{2}(\R^{2})} \|\e^{-itH}u(x)-\e^{-itH}u(y) \|^{k}_{L^{\infty}_{[0,2\pi]}}\text{d}\mu(u)=\int_{\Omega}\big\| \e^{-itH} \gamma^{\om}(x)-\e^{-itH} \gamma^{\om}(y) \big\|^{k}_{L^{\infty}_{[0,2\pi]}}\text{d}\P(\om).
\end{equation}
We have  $\dis \e^{-itH} \gamma^{\om}(x)-\e^{-itH} \gamma^{\om}(y) =\sum_{n=0}^{+\infty}g_{n}(\omega)c_{n} \e^{-it\lambda_{n}}(\phi_{n}(x)-\phi_{n}(y))$. Then by the Khinchin Lemma~\ref{Khin} we get
 \begin{eqnarray*}
\|\e^{-itH} \gamma^{\om}(x)-\e^{-itH} \gamma^{\om}(y)\|_{L^k_{\P}} &\leq&
 C\sqrt{k}\big(\sum_{n=0}^{+\infty} |c_{n}|^{2}|\phi_{n}(x)-\phi_{n}(y)|^2\big)^{1/2}\\
 &=& C\sqrt{k}\big(\sum_{j=1}^{+\infty} \sum_{n\in I(j)}  |c_{n}|^{2}|\phi_{n}(x)-\phi_{n}(y)|^2\big)^{1/2}, 
 \end{eqnarray*}
Recall that  $k\in I(j)=\big\{n\in \N,\;2j\leq \lambda_{n}< 2(j+1)\big\}$ and that  $\# I(j) \sim c j$. Next, by condition \eqref{condi0}, we deduce that 
\begin{equation*}
\|\e^{-itH} \gamma^{\om}(x)-\e^{-itH} \gamma^{\om}(y)\|_{L^k_{\P}}  \leq C \sqrt{k}\Big(\sum_{j=1}^{+\infty}j^{-1}\big( \sum_{\ell \in I(j)} |c_{\ell }|^{2}\big) \sum_{n\in I(j)}   |   \phi_{n}(x)-\phi_{n}(y)|^2\Big)^{1/2}
\end{equation*}

Now we need the following estimate, which is proven in \cite[Lemma 6.1]{IRT} 
\begin{equation*}
 \sum_{n\in I(j)}|\phi_{n}(y)-\phi_{n}(x)|^{2} \leq C |y-x|^{2} j.
\end{equation*}
 Therefore, we obtain 
 \begin{equation*}
\|\e^{-itH} \gamma^{\om}(x)-\e^{-itH} \gamma^{\om}(y)\|_{L^k_{\P}}  \leq C \sqrt{k}|x-y|\|\gamma\|_{L^{2}(\R^{2})},
 \end{equation*}
 and for $k\geq q$, an integration in time and Minkowski yield
  \begin{equation*}
\|\e^{-itH} \gamma^{\om}(x)-\e^{-itH} \gamma^{\om}(y)\|_{L^k_{\P}L^{q}_{[0,2\pi]}}  \leq C \sqrt{k} |x-y|\|\gamma\|_{L^{2}(\R^{2})}.
 \end{equation*}
 However, since the case $q=+\infty$ is forbidden, the previous estimate is not enough to have a control on the $L^{\infty}_{[0,2\pi]}$-norm. To tackle this issue, we claim that for $k\geq q$ we have
 \begin{equation}\label{dtsobo}
\|\e^{-itH} \gamma^{\om}(x)-\e^{-itH} \gamma^{\om}(y)\|_{L^k_{\P}W^{1, q}_{[0,2\pi]}}  \leq C \sqrt{k}  \|\gamma\|_{L^{2}(\R^{2})}.
 \end{equation}
 Then by a usual Sobolev embedding argument, we get that for all $\eps>0$ (by taking $q\gg 1$ large enough)
  \begin{equation*}
\|\e^{-itH} \gamma^{\om}(x)-\e^{-itH} \gamma^{\om}(y)\|_{L^k_{\P}L^{\infty}_{[0,2\pi]}}  \leq C \sqrt{k} |x-y|^{1-\eps} \|\gamma\|_{L^{2}(\R^{2})},
 \end{equation*}
 which in turn implies (using \eqref{mu.defi2}) that 
\begin{equation*}
\mu\big(\,u\in L^{2}(\R^{2})\,:\, \|\e^{-itH}u(x)-\e^{-itH}u(y) \|_{L^{\infty}_{[0,2\pi]}}>K|x-y|^{1-\eps}\,\big) \leq C\e^{-\frac{cK^{2}}{\|\gamma\|^{2}_{L^{2}}}},
\end{equation*}
as we did in the end of the proof of Proposition \ref{regularite}.

Let us now prove \eqref{dtsobo}: We have 
$$\dis \partial_{t}\big(\e^{-itH} \gamma^{\om}(x)-\e^{-itH} \gamma^{\om}(y) \big)=-i\sum_{n=0}^{+\infty}g_{n}(\omega)\lambda_{n}c_{n} \e^{-it\lambda_{n}}(\phi_{n}(x)-\phi_{n}(y)),$$
and with the previous arguments get 
\begin{eqnarray*}
\big\|  \partial_{t}\big(\e^{-itH} \gamma^{\om}(x)-\e^{-itH} \gamma^{\om}(y) \big)\big\|_{L^k_{\P}}  &\leq& C \sqrt{k}\Big(\sum_{j=1}^{+\infty} \big( \sum_{\ell \in I(j)} |c_{\ell }|^{2}\big) \sum_{n\in I(j)}   |   \phi_{n}(x)-\phi_{n}(y)|^2\Big)^{1/2}\\
&\leq& C \sqrt{k} \|\gamma\|_{L^{2}(\R^{2})},
\end{eqnarray*}
where here we have used the Thangavelu/Karadzhov estimate (see \cite[Lemma 3.5]{PRT1}) 
\begin{equation*}
 \sup_{x\in \R^{2}}\sum_{n\in I(j)}|\phi_{n}(x))|^{2} \leq C   .
\end{equation*}
We conclude the proof of \eqref{dtsobo} by integrating in time and using Minkowski.
\end{proof}

We will also need the following  technical result
 
 \begin{lemm}\ph\label{lembeso}
 Let $u_{0}\in \wt{F}_{0}(K)$ and denote by $f(t,x)=\e^{-it H}u_{0}(x)$. Let  $2\leq q<+\infty$ and $g\in L^{q}\big([-T,T]; L^{2}(\R^{2})\big))$. Then if $\eps>0$ is small enough in the definition of $\wt{F}_{0}(K)$
 \begin{equation}\label{bds}
\Big\| \Big( \int_{\R^{2}\times \R^{2}}\frac{|f(t,x)-f(t,y)|^{2}|g(t,x)|^{2}}{|x-y|^{2s+2}}\text{d}x\text{d}y \Big)^{1/2} \Big\|_{L^{q}_{[-T,T]}}\leq CK\|g\|_{L^{q}_{[-T,T]} L^{2}(\R^{2})}.
 \end{equation}
  \end{lemm}
 
 \begin{proof}
 We consider such $f,g$, and we split the integral. On the one hand, we use that $f$ is Lipschitz
 \begin{eqnarray*}
 \int_{|x-y|\leq 1}\frac{|f(t,x)-f(t,y)|^{2}|g(t,x)|^{2}}{|x-y|^{2s+2}}\text{d}x\text{d}y &\leq  &  K^{2} \int_{x\in \R^{2}} |g(t,x)|^{2} \big( \int_{y:\; |x-y|\leq 1}\frac{\text{d}y}{|x-y|^{2s+2\eps}}\big)\text{d}x\\
 & \leq &C K^{2 }\|g(t,\cdot)\|^{2}_{  L^{2}(\R^{2})},
 \end{eqnarray*}
 provided that $s+\eps<1$. We take the $L^{q}_{[-T,T]}$-norm, and we see that this contribution is bounded by the r.h.s. of \eqref{bds}.
 
 On the other hand
  \begin{eqnarray*}
 \int_{|x-y|\geq 1}\frac{|f(t,x)-f(t,y)|^{2}|g(t,x)|^{2}}{|x-y|^{2s+2}}\text{d}x\text{d}y &\leq  & C \|f(t,\cdot)\|^{2}_{  L^{\infty}(\R^{2})}   \int_{x\in \R^{2}} |g(t,x)|^{2} \big( \int_{y:\; |x-y|\geq 1}\frac{\text{d}y}{|x-y|^{2s+2}}\big)\text{d}x\\
 & \leq &C   \|f(t,\cdot)\|^{2}_{  L^{\infty}(\R^{2})}  \|g(t,\cdot)\|^{2}_{  L^{2}(\R^{2})},
 \end{eqnarray*}
if $s>0$. Now we take the $L^{q}_{[-T,T]}$-norm, and use the fact that  $\| f\|_{ L_{[0,2\pi] }^{q} {L}^{\infty  } (\R^2)  }\leq K$ if $\eps<1/q$.
 \end{proof}
 
We now state the main estimates of  this paragraph 

\begin{prop}\ph \label{pointfixe3} There exist $ C > 0 $ and $ \kappa > 0 $ such that if $ u_0 \in \wt{F}_{0}(K)$ for one $ K > 0 $ then for any $ v,v_{1},v_{2} \in {X}_T^s $ and $ 0 < T \leq 1 $,
 \begin{equation}\label{f11}
\bigg\|  \int_0^t  \e^{-i(t-\tau)H}\Big[   \cos(2\tau)^{  p-3 } |  \e^{-i\tau H} u_0 + v |^{p-1}   (   \e^{-i\tau H} u_0 + v ) \Big]\, \text{d}\tau   \bigg\|_{  {X}_T^s }  
 \leq C  T^\kappa   ( K ^p +  \|v\|^p_{  {X}_T^s } ),
\end{equation} 
and
 \begin{align}\label{f22}
& \bigg\|  \int_0^t   \e^{-i(t-\tau)H}  \Big[\cos(2\tau)^{  p-3 }   |  \e^{-i\tau H} u_0 + v_1 |^{p-1}   (  \e^{-i\tau H} u_0 + v_1 )\Big]  \, \text{d}\tau  
 \\ & \hspace*{3cm} -  \int_0^t    \e^{-i(t-\tau)H}\Big[  \cos(2\tau)^{  p-3 } |  \e^{-i\tau H} u_0 + v_2 |^{p-1}   (  \e^{-i\tau H} u_0 + v_2 ) \Big] \, \text{d}\tau  \bigg\|_{  {X}_T^0 }  \nonumber\leq
\nonumber\\ & \leq C  T ^\kappa   \|v_1-v_2\|_{ {X}^0_T }  ( K ^{p-1} +  \|v_1\|^{p-1}_{  {X}_T^s } + \|v_2\|^{p-1}_{  {X}_T^s } ).\nonumber
\end{align} 
\end{prop} 
\begin{proof}
Let $u_{0}\in \wt{F}_{0}(K)$ and set $f=\e^{-is H}u_{0}$.   Let $2<p<3$, then there exists $q\gg1$ so that  $q'(3-p)<1$, which in turn implies $\|  \cos(2s)^{  p-3 }\|_{L^{q'}_{[-T,T]}} \leq C T^{\kappa}$. Next, if $s<1$ is large enough we have by Sobolev
\begin{equation}\label{inj}
\|v\|_{L^{\infty}_{[-T,T]}L^{2(p-1)}(\R^2)}\leq \|v\|_{X^{s}_{T}} \qquad \text{and}\qquad \|v\|_{L^{q(p-1)}_{[-T,T]}L^{\infty}(\R^2)}\leq \|v\|_{X^{s}_{T}}.
\end{equation}
 
$\bullet$ We  prove \eqref{f11}.  From Strichartz and H\"older, we get
 \begin{multline}\label{fir}
 \bigg\|  \int_0^t   \e^{-i(t-s)H}  \Big[\cos(2s)^{  p-3 }  |  f + v |^{p-1}  (   f + v )  \Big]\ \text{d}s   \bigg\|_{  {X}_T^s }  \leq \\
 \begin{aligned}
& \leq  C   \big\|  \cos(2s)^{  p-3 }   |  f + v |^{p-1}  (   f + v )    \big\|_{   L_{[-T,T] }^1 \H^s (\R^2)  }   \\
& \leq  C   \|  \cos(2s)^{  p-3 }\|_{L^{q'}_{[-T,T]}}    \big\|\,    |  f + v |^{p-1}  (   f + v )   \big\|_{   L_{ [-T,T] }^{q} \H^s (\R^2)    }   \\
& \leq  CT^{\kappa}     \big\|\,    |  f + v |^{p-1}  (   f + v )   \big\|_{   L_{ [-T,T] }^{q} \H^s (\R^2)    }.
\end{aligned}
\end{multline}

By using the characterization \eqref{diffsobo}, we will prove that 
\begin{equation}\label{322}
  \big\|\,    |  f + v |^{p-1}  (   f + v )   \big\|_{   L_{ [-T,T] }^{q} \H^s (\R^2)    }\leq  C      ( K ^p +  \|v\|^p_{  {X}_T^s } ).
\end{equation}
The term $ \big\| \<x\>^{s}  |  f + v |^{p-1}  (   f + v )    \big\|_{  L_{ [-T,T] }^{q}  L^{2} (\R^2)   }$ is easily controlled, thus we only detail the contribution of the $H^{s}$ norm.
With \eqref{diff}, it is easy to check that for all $x,y\in \R^{2}$
 \begin{multline*}
\big||f+v|^{p-1}(f+v)(x)- |f+v|^{p-1}(f+v)(y)|\big|   \leq \\
 \leq  C  |v(x) -v(y)|   \big( |v(x)|^{p-1}+  |v(y)|^{p-1}+ |f(x)|^{p-1}+ |f(y)|^{p-1}  \big)\\
+ C    |f(x) -f(y)|     \big( |v(x)|^{p-1}+  |v(y)|^{p-1}+ |f(x)|^{p-1}+ |f(y)|^{p-1} \big).
\end{multline*}
By \eqref{inj} the contribution in $ L_{ [-T,T] }^{q} H^s (\R^2)$  of the first term in the previous expression is 
\begin{equation*}
   \leq C    \big(\big\| f  \big\|^{p-1}_{L^{q(p-1)}_{[-T,T]}L^\infty (\R^2)} +
\big\| v  \big\|^{p-1}_{L^{q(p-1)}_{[-T,T]}L^\infty (\R^2)}\big)\big\|      v  \big\|_{X^{s}_{T}}   \leq C  \big(K^{p-1} +
\big\| v  \big\|^{p-1}_{X^{s}_{T}}\big)\big\|      v  \big\|_{X^{s}_{T}}.
\end{equation*}
To bound the second term, we apply Lemma \ref{lembeso}, which gives a contribution 
\begin{equation*}
\leq  \big(\big\| f  \big\|^{p-1}_{L^{q(p-1)}_{[-T,T]}L^{2(p-1)} (\R^2)} +
\big\| v  \big\|^{p-1}_{L^{q(p-1)}_{[-T,T]}L^{2(p-1)} (\R^2)}\big) K   \leq  C (K^{p-1}+\big\| v  \big\|^{p-1}_{X^{s}_{T}}) K,
 \end{equation*}
which concludes the proof of \eqref{322}.     \ligne

$\bullet$ The proof of \eqref{f22} is in the same spirit, and even easier. We do not write the details.
 \end{proof}

Thanks to the estimates of Proposition \ref{pointfixe3}, for $K>0$ small enough (see the proof of Theorem~\ref{Thm12} for more details) we are able to construct  a unique solution $v\in \mathcal{C}\big([-\pi/4,\pi/4]; L^{2}(\R^{2})\big)$ such that $v\in L^{\infty}\big({[-\pi/4,\pi/4]};\H^{s}(\R^{2})\big)$. By interpolation we deduce that $v\in \mathcal{C}\big({[-\pi/4,\pi/4]};\H^{s'}(\R^{2})\big)$ for all $s'<s$. The end of the proof of Theorem \ref{Thm14} is   similar to the proof of Theorem \ref{Thm12}, using here Lemma~\ref{lemf0}.

\section{Global well-posedness for the cubic equation}\label{Sect4}

\subsection{The case of  dimension $\boldsymbol{ d=3}$} We now turn to the proof of Theorem \ref{theo3D}, which is obtained thanks to the high/low frequency decomposition method of Bourgain  \cite[page 84]{Bou}.\ligne 

Let  $0\leq s< 1$ and fix $\mu=\mu_{\gamma}\in \M^{s}$. For  $ K   \geq 0 $     define the set $ F_{s}( K ) $ as 
 \begin{equation*}
F_{s}(K)=\big\{ w\in \H^{s}(\R^{3})\,:\;   \|w\|_{  \H^{s}(\R^3) }   \leq K, \;   \|w\|_{  L^{4}(\R^3) }\leq K   \;\;\text{and}\;\;   \|  \e^{-itH} w\|_{ L_{ [0,2\pi]}^{1/\eps}\W^{  3/2+s-\eps ,\infty   } (\R^3)  }\leq K \; \big\}.
\end{equation*}
 Then  by Proposition \ref{regularite},  
 \begin{multline}\label{GD2}
 \mu\big( \,(F_{s}(K))^{c}\,\big)\leq  \\
 \leq \mu\big( \, \|w\|_{  \H^{s} }   > K\,  \big)+ \mu\big( \, \|w\|_{  L^{4} }   > K\,  \big)+ \mu\big( \, \|  \e^{-itH} w \|_{ L^{1/\eps}_{ [0,2\pi]} \W^{3/2+  s-\eps , \infty }   }>   K\,  \big)\leq C\e^{-\frac{cK^{2}}{\|\gamma\|^{2}_{\H^{s}}}}.
 \end{multline}
 
 Now we define a smooth version of the usual spectral projector. Let $\chi\in \mathcal{C}_{0}^{\infty}(-1,1)$,  so that $0\leq \chi\leq 1$, with $\chi=1$ on $[-\frac12,\frac12]$. We define  the operators $S_{N}=\chi\big(\frac{H}{N^{2}}\big)$ as 
 
\begin{equation*} 
S_{N}\big(\sum_{n=0}^{+\infty}c_n \phi_{n}\big)=\sum_{n=0}^{+\infty}\chi\big(\frac{\lambda_{n}}{N^{2}}\big)c_n \phi_{n},
\end{equation*}
and we write
\begin{equation*}
v_{N}=S_{N}v,\quad v^{N}=(1-S_{N})v.
\end{equation*}
It is clear that for any $\s\geq 0$ we have $\|S_{N}\|_{\H^{\s}\to \H^{\s}}=1$. Moreover, by \cite[Proposition 4.1]{BTT}, for all $1\leq r\leq +\infty$, $\|S_{N}\|_{L^{r}\to L^{r}}\leq C$, uniformly in $N\geq 1$.

It is straightforward to check that 
\begin{equation}\label{hf}
\|v_{N}\|_{\H^{1}}\leq N^{1-s}\|v\|_{\H^{s}},\quad \|v^{N} \|_{L^{2}}\leq N^{-s}\|v\|_{\H^{s}}.
\end{equation}
Next, let $u_{0}\in F_{s}(N^{\eps})$. By definition of $F_{s}(N^{\eps})$ and \eqref{hf}, $\|u_{0,N}\|_{\H^{1}}\leq N^{1-s} \|u_{0}\|_{\H^{s}}\leq N^{1-s+\eps}$. The nonlinear term of the energy can be controlled by the quadratic term. Indeed
\begin{equation*}
\|u_{0,N}\|^{4}_{L^{4}}\leq CN^{\eps} \leq  N^{2(1-s+\eps)},
\end{equation*}
and thus
\begin{equation}\label{c0}
E(u_{0,N})\leq 2 N^{2(1-s+\eps)}.
\end{equation}
We also have
\begin{equation*}
\|u_{0,N}\|_{L^{2}}\leq \|u_{0}\|_{\H^{s}}\leq   N^{\eps}.
\end{equation*}

 For a nice description of   the stochastic version of the low-high frequency decomposition method we use here, we refer to the introduction of \cite{CO}. To begin with,  we look for a solution $u$ to \eqref{NLS3D} of the form $u=u^{1}+v^{1}$, where 
$u^{1}$ is solution to 
 \begin{equation} \label{NLS.approx}  
  \left\{
      \begin{aligned}
         & i \frac{ \partial u^{1} }{ \partial t } - H u^{1} =   |u^{1}|^{2} u^{1}, \quad (t,x)\in \R\times \R^{3},
       \\  &  u^{1}(0)  =u_{0,N},
      \end{aligned}
    \right.
\end{equation}
and where $v^{1}=\e^{-itH}u^{N}_{0}+w^{1}$ satisfies
 \begin{equation} \label{fluctu}  
  \left\{
      \begin{aligned}
         & i \frac{ \partial w^{1} }{ \partial t } - H w^{1} =   |w^{1}+\e^{-itH}u^{N}_{0}+u^{1}|^{2} \big(w^{1}+\e^{-itH}u^{N}_{0}+u^{1}\big)-|u^{1}|^{2} u^{1}, \quad (t,x)\in \R\times \R^{3},
       \\  &  w^{1}(0)  =0.
      \end{aligned}
    \right.
\end{equation}

Since equation \eqref{NLS.approx} is $\H^{1}-$subcritical, by the usual deterministic arguments, there exists a unique global solution $u^{1}\in \mathcal{C}\big(\R,\H^{1}(\R^{3})\big)$.

We now turn to \eqref{fluctu}, for which we have the next local existence result. 

\begin{prop}\ph\label{prop.fluctu}
Let $0<s<1$ and $\mu=\mu_{\gamma}\in \M^{s}$. Set $T=N^{-4(1-s)-\eps}$ with $\eps>0$. Assume that $E(u^{1})\leq 4N^{2(1-s+\eps)}$ and $\|u^{1}\|_{L^{\infty}_{[0,T]}L^{2}}\leq 2N^{\eps}$. Then 
\begin{enumerate}[(i)]
\item There exists a set $\Sigma^{1}_{T}\subset \H^{s}$ which only depends on $T$ so that 
$$\mu(\Sigma^{1}_{T})\geq 1-C\exp\big(-cT^{-\delta}\|\gamma\|^{-2}_{\H^{s}(\R^{3})}\big),$$
with some $\delta>0$.
\item For all $u_{0}\in \Sigma^{1}_{T}$ there exists a unique solution $w^{1} \in \mathcal{C}\big([0,T],\H^{1}(\R^{3})\big)$ to equation \eqref{fluctu} which satisfies the bounds
\begin{equation}\label{5.10}
\|w^{1}\|_{L^{\infty}_{[0,T]}\H^{1}}\leq CN^{\beta(s)+c\eps},
\end{equation}
with 
 \begin{equation}\label{beta} 
\beta(s)=\left\{\begin{array}{ll} 
-5/2,\quad &\text{if} \quad 0\leq s\leq 1/2, \\[6pt]  
2s-7/2,  &\text{if} \quad 1/2\leq s\leq 1,
\end{array} \right.
\end{equation}
and
\begin{equation} \label{5.11}
\|w^{1}\|_{L^{\infty}_{[0,T]}L^{2}}\leq CN^{-9/2+2s+c\eps}.
\end{equation}
\end{enumerate}
\end{prop} 

\begin{proof}
 In the next lines, we write $C^{a+}=C^{a+b\eps}$, for some absolute quantity $b>0$.
Since $d=3$, for $T>0$, we define the space $X^{1}_{T}=L^{\infty} \big([0,T]; \H^{1}(\R^{3})\big)\bigcap L^{2}\big([0,T]; \W^{1,6}(\R^{3})\big)$. Let $\eps>0$, and define $\Sigma^{1}_{T}=F_{s}(N^{\eps})$. By \eqref{GD2} and the choice $T=N^{-4(1-s)-\eps}$, the set $\Sigma^{1}_{T}$ satisfies $(i)$. 

Let $u_{0}\in \Sigma^{1}_{T}$. To simplify the notations in the proof, we write $w=w^{1}$, $u=u^{1}$ and $f=\e^{-itH}u^{N}_{0}$. We define the map
\begin{equation}\label{mapL}
L ( w) =    \mp i \int_0^t   \e^{-i(t-s)H}\big( |   f +u+ w |^{2} (    f +u+ w ) -|u|^{2}u\big)(s) \text{d}s.
\end{equation}
First we prove \eqref{5.10}. By Strichartz \eqref{Stri}
\begin{equation}\label{rh}
\|L(w)\|_{X^{1}_{T}}\leq C \big\||   f +u+ w |^{2} (    f +u+ w ) -|u|^{2}u\big\|_{L^{1}_{T}\H^{1}+L^{2}\W^{1,6/5}}.
\end{equation}
By estimating the contribution of every term, we now prove that 
\begin{equation}\label{lundi}
\|L(w)\|_{X^{1}_{T}}\leq CN^{\beta(s)+}+N^{0-}\|w\|_{X^{1}_{T}}+N^{-2(1-s)+}\|w\|^{3}_{X^{1}_{T}},
\end{equation}
where $\beta(s)< (1-s)$ is as in the statement. It is enough to prove that $L$ maps a ball   of size $CN^{\beta(s)+}$ into itself, for times $T=N^{-4(1-s)-\eps}$. With similar arguments can show that $L$ is a contraction (we do not write the details) and get $w$ which satisfies \eqref{5.10}.\ligne

 Observe that the complex conjugation is harmless with respect to the considered norms, thus we can forget it. By definition of $\Sigma^{1}_{T}=F_{s}(N^{\eps})$ and \eqref{hf} we have the estimates which will be used in the sequel: for all $\s<3/2$
\begin{equation}\label{512}
\|f\|_{L^{\infty}_{T}L^{2}}\leq CN^{-s+\eps},\quad \|H^{\s/2}f\|_{L^{\infty-}_{T}L^{\infty}}\leq CN^{\s-3/2-s+2\eps}.
\end{equation}
Let us detail the proof of the second estimate. 
\begin{eqnarray*}
 \|H^{\s/2}f\|_{L^{\infty-}_{T}L^{\infty}}&=&N^{\s} \|\big(\frac{H}{N^{2}}\big)^{\s/2} \big(1-\chi\big(\frac{H}{N^{2}}\big)\big)\e^{-itH}u_{0}\|_{L^{\infty-}_{T}L^{\infty}}\\
 &\leq&C N^{\s} \|\big(\frac{H}{N^{2}}\big)^{(3/2+s-\eps)/2} \big(1-\chi\big(\frac{H}{N^{2}}\big)\big)\e^{-itH}u_{0}\|_{L^{\infty-}_{T}L^{\infty}}\\
  &\leq &C N^{\s-3/2-s+\eps} \|\e^{-itH}u_{0}\|_{L^{\infty-}_{T}\W^{3/2+s-\eps,\infty}}\\
    &\leq &C N^{\s-3/2-s+2\eps},
\end{eqnarray*}
where we have used that $x^{\s/2}(1-\chi(x))\leq C x^{(3/2+s-\eps)/2}(1-\chi(x))$.

Observe also that by assumption
\begin{equation*} 
\|u\|_{L^{\infty}_{T}L^{2}}\leq CN^{\eps}  ,\quad\|u\|_{L^{\infty}_{T}\H^{1}}\leq CN^{1-s+\eps},\quad \|u\|_{L^{\infty}_{T}L^{4}}\leq CN^{(1-s+\eps)/2}.
\end{equation*}
~

We now estimate each term in the r.h.s. of \eqref{rh}.\\
$\bullet$ Source terms:  Observe that $L^{4/3}_{T}\W^{1,3/2}\subset L^{1}_{T}\H^{1}+L^{2}\W^{1,6/5}$, then by H\"older and \eqref{512}
\begin{eqnarray*}
\|fu^{2}\|_{L^{1}_{T}\H^{1}+L^{2}\W^{1,6/5}}&\leq& C \| f uH^{1/2}u\|_{L^{4/3}_{T}L^{3/2}}+ C\|u^{2}H^{1/2}f\|_{L^{1}_{T}L^{2}} \nonumber \\
&\leq& C T^{3/4-}\|u\|_{L^{\infty}_{T}\H^{1}}\|u\|_{L^{\infty}_{T}L^{6}}\|f\|_{L^{\infty-}_{T}L^{\infty}}+   C T^{1-}\|u\|^{2}_{L^{\infty}_{T}L^{4}} \|H^{1/2}f\|_{L^{\infty-}_{T}L^{\infty}}\nonumber\\
&\leq	 &CN^{-5/2+}+CN^{-7/2+2s+}\leq CN^{\beta(s)+},
\end{eqnarray*}
where we have set $\beta(s)=\max(-5/2,-7/2+2s)$ which is precisely \eqref{beta}.  Similarly,
\begin{eqnarray*}
\|f^{2}u\|_{L^{1}_{T}\H^{1}}&\leq& C \|f^{2}H^{1/2}u\|_{L^{1}_{T}L^{2}}+ C\|ufH^{1/2}f\|_{L^{1}_{T}L^{2}} \nonumber \\
&\leq& C T^{1-}\|u\|_{L^{\infty}_{T}\H^{1}}\|f\|^{2}_{L^{\infty-}_{T}L^{\infty}}+   C T^{1-}\|u\|_{L^{\infty}_{T}L^{2}}\|f\|_{L^{\infty-}_{T}L^{\infty}} \|H^{1/2}f\|_{L^{\infty-}_{T}L^{\infty}}\nonumber\\
&\leq	 &CT^{1-} N^{-2-3s+}+CT^{1-}N^{-2-2s+}\leq CN^{-6+2s+}\leq CN^{\beta(s)+}.
\end{eqnarray*}
Finally,
\begin{eqnarray*}
\|f^{3}\|_{L^{1}_{T}\H^{1}}&\leq& C \|f^{2}H^{1/2}f\|_{L^{1}_{T}L^{2}}  \leq CT^{1-} \|H^{1/2}f\|_{L^{\infty-}_{T}L^{\infty}}\|f\|_{L^{\infty-}_{T}L^{\infty}}\|f\|_{L^{\infty}_{T}L^{2}}\nonumber\\
&\leq	 &CT^{1-}N^{-1/2-s+}N^{-3/2-s+}N^{-s+}\leq CN^{-6+s+}\leq CN^{\beta(s)+}.
\end{eqnarray*}
$\bullet$ Linear terms in $w$:
\begin{eqnarray*}
\|wf^{2}\|_{L^{1}_{T}\H^{1}}&\leq& C \| f^{2}H^{1/2}w\|_{L^{1}_{T}L^{2}}+ C\|wf H^{1/2}f\|_{L^{1}_{T}L^{2}} \nonumber \\
&\leq& C T^{1-}\|f\|^{2}_{L^{\infty-}_{T}L^{\infty}}\|w\|_{L^{\infty}_{T}\H^{1}}+   C T^{1-}\|w\|_{L^{\infty}_{T}L^{2}} \|f\|_{L^{\infty-}_{T}L^{\infty}}\|H^{1/2}f\|_{L^{\infty-}_{T}L^{\infty}}\nonumber\\
&\leq	 &C  N^{-6+2s+}\|w\|_{X^{1}_{T}}\leq CN^{0-}\|w\|_{X^{1}_{T}}.
\end{eqnarray*}
Use that  $\|w\|_{L^{4/3+}_{T}L^{\infty-}}\leq CT^{1/2-} \|w\|_{L^{4}_{T}L^{\infty-}}\leq CT^{1/2-}\|w\|_{L^{4}_{T}\W^{1,3}}$ and $X^{1}_{T}\subset L^{4}\big([0,T];\W^{1,3}\big)$
\begin{eqnarray*}
\|wu^{2}\|_{L^{1}_{T}\H^{1}+L^{2}\W^{1,6/5}}&\leq& C \| u^{2}H^{1/2}w\|_{L^{1}_{T}L^{2}}+ C\|wu H^{1/2}u\|_{L^{4/3+}_{T}L^{3/2-}} \nonumber \\
&\leq& C\|u\|^{2}_{L^{4}_{T}L^{6}}\|w\|_{L^{2}_{T}\W^{1,6}}+   C \|w\|_{L^{4/3+}_{T}L^{\infty-}} \|u\|_{L^{\infty}_{T}L^{6}}\|u\|_{L^{\infty}_{T}\H^{1}}\nonumber\\
&\leq	 &CT^{1/2-}\|u\|^{2}_{L^{\infty}_{T}\H^{1}}\|w\|_{X^{1}_{T}}\leq CN^{0-}\|w\|_{X^{1}_{T}}.
\end{eqnarray*}
$\bullet$ The cubic term in $w$: by Sobolev and $X^{1}_{T}\subset L^{4-}\big([0,T];\W^{1,3+}\big)\subset L^{4-}\big([0,T]; L^{\infty}\big) $
\begin{eqnarray*}
\|w^{3}\|_{L^{1}_{T}\H^{1}}&\leq& C \|w^{2}H^{1/2}w\|_{L^{1}_{T}L^{2}}\leq C\|w\|_{L^{\infty}_{T}\H^{1}}\|w\|^{2}_{L^{2}_{T}L^{\infty}} \nonumber\\
&\leq& CT^{1/2-}\|w\|^{3}_{X^{1}_{T}}\leq CN^{-2(1-s)+}\|w\|^{3}_{X^{1}_{T}}.
\end{eqnarray*}
$\bullet$ Quadratic terms in $w$: with similar arguments, we check that  they are controlled by the previous ones.

This completes the proof of \eqref{lundi}. Hence for all $u_{0}\in \Sigma^{1}_{T}$, $L$ has a unique fixed point $w$.\ligne

Let $w\in X^{1}_{T}$ be defined this way, and let us prove that $\|w\|_{X^{0}_{T}}\leq C N^{-9/2+2s+}$, which will imply~\eqref{5.11}. By the Strichartz inequality \eqref{Stri}
\begin{equation*}
\|w\|_{X^{0}_{T}} \leq C \big\||   f +u+ w |^{2} (    f +u+ w ) -|u|^{2}u\big\|_{L^{1}_{T}L^{2}+L^{2}L^{6/5}}.
\end{equation*}
As previously, the main contribution in the source term is 
\begin{equation*}
\|fu^{2}\|_{L^{1}_{T}L^{2}} \leq  T^{1-}\|u\|^{2}_{L^{\infty}_{T}L^{4}}\|f\|_{L^{\infty-}_{T}L^{\infty}} \leq C N^{-4(1-s)+1-s-3/2-s+}=CN^{-9/2+2s+}.
\end{equation*}
For the cubic term we write
\begin{eqnarray*}
\|w^{3}\|_{L^{1}_{T}L^{2}}&\leq& \|w\|_{L^{\infty}_{T}L^{2}}\|w\|^{2}_{L^{2}_{T}L^{\infty}} \nonumber\\
&\leq& CT^{1/2-}\|w\|_{L^{\infty}_{T}L^{2}}\|w\|^{2}_{X^{1}_{T}}\nonumber\\
&\leq& CN^{-2(1-s)+\beta(s)+}\|w\|_{L^{\infty}_{T}L^{2}}\leq CN^{0-}\|w\|_{X^{0}_{T}},
\end{eqnarray*}
which gives a control by the linear term.

The other terms are controlled with similar arguments, and we leave the details to the reader. This finishes the proof of Proposition \ref{prop.fluctu}.
\end{proof}

\begin{lemm}\ph\label{lem.energy}
Under the assumptions of Proposition \ref{prop.fluctu}, for all $u_{0}\in \Sigma^{1}_{T}$ we have 
\begin{equation*}
\big|E\big(\,u^{1}(T)+w^{1}(T)\,\big)-E\big(\,u^{1}(T)\,\big)\big|\leq CN^{1-s+\beta(s)+}.
\end{equation*}
\end{lemm}

\begin{proof}
Write $u=u^{1}$ and $w=w^{1}$. A direct expansion and H\"older  give
\begin{multline*}
\big|E\big(\,u(T)+w(T)\,\big)-E\big(\,u(T)\,\big)\big|\leq \\
\leq 2\|u\|_{L^{\infty}_{T}\H^{1}}\|w\|_{L^{\infty}_{T}\H^{1}} +\|w\|^{2}_{L^{\infty}_{T}\H^{1}}
+C\|w\|_{L^{\infty}_{T}L^{4}}\|u\|^{3}_{L^{\infty}_{T}L^{4}}+C\|w\|^{4}_{L^{\infty}_{T}L^{4}}.
\end{multline*}
$\bullet$ Since $\beta(s)\leq (1-s)$, we directly have
\begin{equation*}
2\|u\|_{L^{\infty}_{T}\H^{1}}\|w\|_{L^{\infty}_{T}\H^{1}}+\|w\|^{2}_{L^{\infty}_{T}\H^{1}}\leq CN^{1-s+\beta(s)+}.
\end{equation*}
$\bullet$ By Sobolev and Proposition \ref{prop.fluctu}
\begin{equation}\label{w4}
\|w\|_{L^{\infty}_{T}L^{4}}\leq C \|w\|_{L^{\infty}_{T}\H^{3/4}}\leq     C\|w\|^{1/4}_{L^{\infty}_{T}L^{2}}\|w\|^{3/4}_{L^{\infty}_{T}\H^{1}}\leq C N^{\eta(s)+},
\end{equation}
with $\eta(s)=\max(-3+s/2,-15/4+2s)\leq (1-s+\beta(s))/3$. Hence
\begin{equation*}
\|w\|^{3}_{L^{\infty}_{T}L^{4}}\leq CN^{1-s+\beta(s)+}.
\end{equation*}
$\bullet$ From the  bounds    $\|u\|_{L^{\infty}_{T}L^{4}}\leq CN^{(1-s)/2}$ and \eqref{w4}, we infer
\begin{equation*}
\|w\|_{L^{\infty}_{T}L^{4}}\|u\|^{3}_{L^{\infty}_{T}L^{4}}\leq CN^{\delta(s)+},
\end{equation*}
where $\delta(s)=\max(-3+s/2,-15/4+2s)\leq 1-s+\beta(s)$ (with equality when $0<s\leq 1/2$).

This completes the proof.
\end{proof}

With the results of Proposition \ref{prop.fluctu} and Lemma \ref{lem.energy},  we are able to iterate the argument. At time $t=T$, write $u=u^{2}+v^{2}$ where 
$u^{2}$ is solution to 
 \begin{equation} \label{NLS.approx2}  
  \left\{
      \begin{aligned}
         & i \frac{ \partial u^{2} }{ \partial t } - H u^{2} =   |u^{2}|^{2} u^{2}, \quad (t,x)\in \R\times \R^{3},
       \\  &  u^{2}(T)  =u^{1}(T)+w^{1}(T)\in \H^{1}(\R^{3}),
      \end{aligned}
    \right.
\end{equation}
and where $v^{2}=\e^{-itH}u^{N}_{0}+w^{2}$ satisfies
 \begin{equation*}  
  \left\{
      \begin{aligned}
         & i \frac{ \partial w^{2} }{ \partial t } - H w^{2} =   |w^{2}+\e^{-itH}u^{N}_{0}+u^{2}|^{2} \big(w^{2}+\e^{-itH}u^{N}_{0}+u^{2}\big)-|u^{2}|^{2} u^{2}, \quad (t,x)\in \R\times \R^{3},
       \\  &  w^{2}(T)  =0.
      \end{aligned}
    \right.
\end{equation*}
By Proposition \ref{prop.fluctu}, $w^{1}(T)\in \H^{1}(\R^{3})$, thus \eqref{NLS.approx2} is globally well-posed. Then, thanks to  Lemma~\ref{lem.energy}, by the conservation of the energy 
\begin{equation*}
E(u^{2})=E(u^{1}(T)+w^{1}(T))\leq 4N^{2(1-s+\eps)},
\end{equation*}
and by the conservation of the mass
\begin{equation*}
\|u^{2}\|_{L^{\infty}_{T}L^{2}}=\|u^{1}(T)+w^{1}(T)\|_{L^{2}}\leq 2N^{\eps}.
\end{equation*}
Therefore there exists a set $\Sigma^{2}_{T}\subset \H^{s}$ with 
$$\mu(\Sigma^{2}_{T})\geq 1-C\exp\big(-cT^{-\delta}\|\gamma\|^{-2}_{\H^{s}}\big),$$
and so that for all $u_{0}\in \Sigma^{2}_{T}$, there exists a unique $w^{2}\in \mathcal{C}\big([T,2T],\H^{1}(\R^{3})\big)$ which satisfies the result of   Proposition \ref{prop.fluctu}, with the same $T>0$. Here we use crucially that the large deviation bounds of Proposition \ref{regularite} are invariant under time shift $\tau$.\ligne

{\bf Iteration of the argument:} Fix a time $A>0$. We can iterate the previous argument and construct $u^{j}$, $v^{j}$ and $w^{j}$ for $1\leq j\leq \lfloor A/T\rfloor$ so that \vspace{5pt}

\begin{itemize}
\item[$\bullet$] The function $u^{j}$ is solution to $\eqref{NLS.approx2}$ with initial condition 
$$u^{j}(t=(j-1)T)=u^{j-1}((j-1)T)+w^{j-1}((j-1)T)\;;$$
\item[$\bullet$]  We set $v^{j}(t)=\e^{-itH}u^{N}_{0}+w^{j}(t)$ where the function $w^{j}$ is solution to 
 \begin{equation*}   
  \left\{
      \begin{aligned}
         & i \frac{ \partial w^{j} }{ \partial t } - H w^{j} =   |w^{j}+\e^{-itH}u^{N}_{0}+u^{j}|^{2} \big(w^{j}+\e^{-itH}u^{N}_{0}+u^{j}\big)-|u^{j}|^{2} u^{j}, \quad (t,x)\in \R\times \R^{3},
       \\  &  w^{j}((j-1)T)  =0.
      \end{aligned}
    \right.
\end{equation*}
\end{itemize}
This enables to define a unique solution $u$ to the initial problem \eqref{NLS3D} defined by $u(t)=u^{j}(t)+v^{j}(t)$ for $t\in[(j-1)T,jT]$, with $1\leq j\leq \lfloor A/T\rfloor$ provided that $\dis u_{0}\in \Gamma^{A}_{T}$, where 
\begin{equation*}
\Gamma^{A}_T:= \bigcap_{j=1}^{\lfloor A/T\rfloor} \Sigma^{j}_{T}.
\end{equation*}
Thanks to the exponential bounds, we have $\mu\big((\Gamma^{A}_{T})^{c}\big)\leq C\exp\big(-cT^{-\delta/2}\|\gamma\|^{-2}_{\H^{s}}\big)$, with $T=N^{-4(1-s)-\eps}$. \medskip

{\bf Uniform bounds on the energy and the mass:}  It remains to check whether $E(u^{j})\leq 4N^{2(1-s+\eps)}$ and $\|u^{j}\|_{L^{2}(\R^{3})}\leq 2N^{\eps}$ for all $1\leq j\leq \lfloor A/T\rfloor$. By Lemma \ref{lem.energy}, for $T=N^{-4(1-s)-}$
\begin{equation}\label{borne}
E(u^{j})\leq E(u_{0,N})+CAT^{-1}N^{1-s+\b(s)+} \leq 2N^{2(1-s+\eps)}+CAN^{\b(s)+5(1-s)+},
\end{equation}
which satisfies the prescribed bound iff \,$3(1-s)+\beta(s)<0$.  
\\[5pt]
$\bullet$ Let $1/2\leq s\leq 1$, then the condition is
$$3(1-s)+2s-\frac72<0\quad \text{iff}\quad s>-\frac12,$$
which is satisfied. \\
$\bullet$ Let $0\leq s\leq 1/2$, then the condition is
$$3(1-s)-\frac52<0\quad \text{iff}\quad s>\frac{1}{6}.$$
The same argument applies to control $\|u^{j}\|_{L^{2}}$.\medskip

{\bf Optimisation on $\boldsymbol {N\geq1}$:} If $1/6<s<1$, we optimise in \eqref{borne} with the choice $N$ so that  $A\sim cN^{-3(1-s)-\beta(s)}$, and get that for $1\leq j\leq \lfloor A/T\rfloor$
\begin{equation*}
E(u^{j})\leq CA^{c_{s}+},
\end{equation*}
with 
 \begin{equation}\label{defcs} 
c_{s}=\left\{\begin{array}{ll} 
\dis 2(1-s)/(6s-1) ,\quad &\text{if} \quad 1/6< s\leq 1/2, \\[6pt]  
\dis {2(1-s)}/{(2s+1)},  &\text{if} \quad 1/2\leq s\leq 1.
\end{array} \right.
\end{equation}
Denote by $\Gamma^{A}=\Gamma^{A}_{T}$ the set defined with the previous choice of $N$ and  $T=N^{-4(1-s)-\eps}$.  
\begin{lemm}\ph\label{lem.control}
Let $1/6<s<1$. Then for all $A\in \N$ and all $u_{0}\in \Gamma^{A}$, there exists a unique solution to \eqref{NLS3D} on $[0,A]$ which reads 
\begin{equation*}
u(t)=\e^{-itH}u_{0}+w(t),\quad w\in \mathcal{C}\big([0,A],\H^{1}(\R^{3})\big),
\end{equation*}
and so that  
\begin{equation*}
\sup_{t\in [0,A]}E\big(w(t)\big)\leq CA^{c_{s}+}.
\end{equation*}
\end{lemm}

\begin{proof}
On the time interval $[(j-1)T,jT]$ we have $u=u^{j}+v^{j}$ where $v^{j}=\e^{-itH}u^{N}_{0}+w^{j}$ and $u^{j}=\e^{-itH}u_{0,N}+z^{j}$, for some $z_{j}\in \mathcal{C}\big([0,+\infty[,\H^{1}(\R^{3})\big)$. Therefore, if we define $w\in \mathcal{C}\big([0,A],\H^{1}(\R^{3})\big)$ by $w(t)=z^{j}(t)+w^{j}(t)$ for $t\in [(j-1)T,jT]$ and $1\leq j\leq \lfloor A/T\rfloor$ we get $u(t)=\e^{-itH}u_{0}+w(t)$ for all $t\in [0,A]$. Next, for $t\in [(j-1)T,jT]$
\begin{equation*}
E(w(t))\leq CE(z^{j})+CE(w^{j})\leq  CE(u^{j})+CE(\e^{-itH}u_{0,N})+CE(w^{j})\leq C A^{c_{s}+},
\end{equation*}
which was the claim.
\end{proof}

We are now able to complete the proof of Theorem \ref{theo3D}. Set 
$$  \Theta =\bigcap_{k=1}^{+\infty} \bigcup_{A\geq k} \Gamma^{A} \;\;\mbox{ and } \;\;\Sigma = \Theta + \H^1.$$

$ \bullet $ We have $  \mu(\Theta)= \underset{ k \rightarrow \infty}{\lim}  \mu \Big( \underset{A\geq k}{\bigcup} \Gamma^{A} \Big) $ and $ \mu \Big( \underset{A\geq k}{\bigcup} \Gamma^{A} \Big) \geq 1-c\exp(-k^{\delta}\|\gamma\|^{-2}_{\H^{s}})$. So $ \mu( \Theta ) = 1 $, then $ \mu ( \Sigma ) =1 $.

$ \bullet $ By definition, for all $u_{0}\in \Theta $, there exists a unique global solution to \eqref{NLS3D} which reads
\begin{equation*}
u(t)=\e^{-itH}u_{0}+w(t),\quad w\in \mathcal{C}\big(\R,\H^{1}(\R^{3})\big).
\end{equation*}
Then by Lemma~\ref{lem.control} for all $u_{0}\in \Theta $, there exists a unique $w\in \mathcal{C}\big([0,+\infty[,\H^{1}(\R^{3})\big)$ which satisfies for all $N$ the bound
\begin{equation*}
\sup_{t\in [0,N]}E\big(w(t)\big)\leq CN^{c_{s}+}.
\end{equation*}
Now, if $ U_0 \in \Sigma $ then $ U_0 = u_0 + v  $ with $ u_0 \in \Theta $, $ v \in \H^1 $ and we can use the method of Proposition~\ref{prop.fluctu}, Lemma \ref{lem.energy} and Lemma \ref{lem.control} with $ U_{0,N} $ replaced by $ u_{0,N}+ v $. And the set $ \Sigma $ checks properties $ (i) $ and~$(ii)$.

$ \bullet $ Coming back to the definition of $ \Sigma_T^j $, we have for all $ t \in  \R $, $ \e^{-itH}(\Sigma_T^j) = \Sigma_T^j $ then $ \e^{-itH}(\Theta) = \Theta $. Finally, thanks to the property $ (i) $, the set $ \Sigma $ is invariant under the dynamics and the property $ (iii) $ is satisfied.

\subsection{The case of  dimension $\boldsymbol {d=2}$}
 In this section, we prove the Theorem \ref{theo2D}. The proof is analogous to Theorem \ref{theo3D} in a simpler context, that is why, we only explain the key estimates.

$ \bullet $ According to Proposition \ref{regularite}, we set
 \begin{equation*}
F_{s}(K)=\big\{ w\in \H^{s}(\R^{2})\,:\;\;   \|w\|_{  \H^{s}(\R^2) }   \leq K , \quad   \|w\|_{  L^{4}(\R^2) }\leq K   \;\;\text{and}\;\;   \|  \e^{-itH} w\|_{ L_{[0,2\pi]}^{1/\eps} \W^{  1+s-\eps ,\infty   } (\R^2)  }\leq K \; \big\},
\end{equation*}
and we fix $ u_0  \in F_{s}(N^\eps)$. 
\\Then, if $ f = \e^{-itH} u_0^N $, we have  
\begin{equation*}
\|f\|_{L^{\infty}_{[0,2\pi]}L^{2}}\leq CN^{-s+\eps},\quad \|H^{\s/2}f\|_{L^{\infty-}_{ [0,2\pi]},L^{\infty}}\leq CN^{\s-1-s+\eps}.
\end{equation*}
 
$ \bullet $ In  Proposition \ref{prop.fluctu}, we can choose $ T = N^{-2(1-s)-\epsilon} $ to have
\begin{equation*} 
\|u^1\|_{L^{\infty}_{T}L^{2}}\leq CN^{\eps}  \quad \mbox{ and }\quad \|u^1\|_{L^{\infty}_{T}\H^{1}}\leq CN^{1-s+\eps}.
\end{equation*} 
Moreover, as $ u_0 \in F_{s}(N^\eps)$, we obtain
\begin{equation*}
\|u_{0,N}\|_{L^{4}}\leq CN^{\eps}.
\end{equation*}
Hence, we establish
\begin{align*}
E(u^1) & = \|u^1\|^2_{\H^1(\R^2)} + \frac{1}{2} \|u^1\| _{L^4(\R^2)}^4 = \|u_{0,N}\|^2_{\H^1(\R^2)} + \frac{1}{2} \|u_{0,N}\| _{L^4(\R^2)}^4
\\ & \leq N^{2(1-s+\eps)} + C  N^{4\eps}  \leq 4 N^{2(1-s+\eps)},
\end{align*}
and
\begin{equation*}
\|u^1\|_{L^{\infty}_{T}L^{4}}\leq C N^{(1-s+\eps)/2}.
\end{equation*}

$ \bullet $ In Proposition \ref{prop.fluctu}, we obtain
\begin{equation*}
\|w^1\|_{ L^\infty_{[0,T]}  \H^1  } \leq C N^{-1+} \quad \mbox{ and }\quad \|w^1\|_{ L^\infty_{[0,T]}  L^2  } \leq C N^{-2+}.
\end{equation*}
The proof is essentially the same: We define the map $L$ as in \eqref{mapL}. For the first estimate,  we prove that
\begin{equation*} 
\|L(w)\|_{X^{1}_{T}}\leq CN^{-1+}+N^{0-}\|w\|_{X^{1}_{T}}+N^{-2(1-s)+}\|w\|^{3}_{X^{1}_{T}}.
\end{equation*}
 We only give details of source terms.
\begin{eqnarray*}
\|fu^{2}\|_{ L^{1+}_{T}\W^{1,2-}   +L^{1}_{T}\H^{1}} & \leq& C \| f uH^{1/2}u\|_{L^{1+}_{T}L^{2-} } + C\|u^{2}H^{1/2}f\|_{L^{1}_{T}L^{2}} \nonumber \\
& \leq& C T^{1-}\|u\|_{L^{\infty}_{T}\H^{1}}\|u\|_{L^{\infty}_{T}L^{\infty-}}\|f\|_{L^{\infty-}_{T}L^{\infty}}+   C T^{1-}\|u\|^{2}_{L^{\infty}_{T}L^{4}} \|H^{1/2}f\|_{L^{\infty-}_{T}L^{\infty}}\nonumber\\
&\leq	 &C  T^{1-} \max( N^{1-3s+} , N^{1-2s+}) \leq C T^{1-} N^{1-2s+} \leq C N^{-1+}.
\end{eqnarray*}
Similarly,
\begin{eqnarray*}
\|f^{2}u\|_{L^{1}_{T}\H^{1}}&\leq& C \|f^{2}H^{1/2}u\|_{L^{1}_{T}L^{2}}+ C\|ufH^{1/2}f\|_{L^{1}_{T}L^{2}} \nonumber \\
&\leq& C T^{1-}\|u\|_{L^{\infty}_{T}\H^{1}}\|f\|^{2}_{L^{\infty-}_{T}L^{\infty}}+   C T^{1-}\|u\|_{L^{\infty}_{T}L^{2}}\|f\|_{L^{\infty-}_{T}L^{\infty}} \|H^{1/2}f\|_{L^{\infty-}_{T}L^{\infty}}\nonumber\\
&\leq	 &C T^{1-} \max( N^{-1-3s+} , N^{-1-2s+}) \leq C T^{1-} N^{-1-2s+} \leq C N^{-3+}\leq CN^{-1+}.
\end{eqnarray*}
Finally,
\begin{eqnarray*}
\|f^{3}\|_{L^{1}_{T}\H^{1}}&\leq& C \|f^{2}H^{1/2}f\|_{L^{1}_{T}L^{2}}  \leq CT^{1-} \|H^{1/2}f\|_{L^{\infty-}_{T}L^{\infty}}\|f\|_{L^{\infty-}_{T}L^{\infty}}\|f\|_{L^{\infty}_{T}L^{2}}\nonumber\\
&\leq	 &CT^{1-}N^{-s+}N^{-1-s+}N^{-s+}\leq CN^{-3-s+}\leq CN^{-1+}.
\end{eqnarray*}
 
$ \bullet $ Analogously to Lemma \ref{lem.energy}, we obtain 
\begin{equation*}
\big|E\big(\,u^{1}(T)+w^{1}(T)\,\big)-E\big(\,u^{1}(T)\,\big)\big|\leq CN^{-s+},
\end{equation*}
 because, here $ \beta(s) = 1- $, and the estimates on $ u^1 $ are the same that in  dimension $d=3$.
 
$ \bullet $ Finally, the globalization argument holds if (\ref{borne}) is satisfied, that is to say
\begin{equation*}
C A T^{-1} N^{-s+} \leq 4 N^{2(1-s)+}, 
\end{equation*}
which is equivalent to $2(1-s)-s < 2(1-s)$,  hence $s>0$. In this case, we set $A\sim cN^{s}$ and we get that for $0\leq t\leq A$ 
\begin{equation*}
E(w(t))\leq C A^{c_{s}+},\quad \text{with}\;\; c_{s}=\frac{1-s}{s}.
\end{equation*}
Theorem \ref{theo2D} follows.

 
\end{document}